\documentclass[11pt,twoside, a4paper, english, reqno]{amsart}
\usepackage[dvips]{epsfig}
\usepackage{amscd}
\usepackage{amssymb}
\usepackage{amsthm}
\usepackage{amsmath}
\usepackage{latexsym}
\usepackage{amsaddr}

\usepackage{upref}
\usepackage{bm}
\usepackage{color}
\usepackage{hyperref}

\hypersetup{linkcolor=blue, colorlinks=true
,citecolor = red}
\theoremstyle{plain}
\newtheorem{thm}{Theorem}[section]
\newtheorem{thma}{Theorem}

\theoremstyle{plain}
\newtheorem{lem}[thm]{Lemma}

\newtheorem{prop}[thm]{Proposition}

\theoremstyle{definition}
\newtheorem{defi}{Definition}[section]
\newtheorem{rem}{Remark}

\newcommand{\be} {\begin{equation}}
\newcommand{\ee} {\end{equation}}

\newcommand{\R} {\mathbb{R}}
\newcommand{\brho} {\bar{\rho}}
\newcommand{\bu} {\bar{u}}

\newcommand{\rt}{\mathbb{R}^{2}}

\newcommand{\authorfootnotes}{\renewcommand\thefootnote{\@fnsymbol\c@footnote}}%

\numberwithin{equation}{section} \allowdisplaybreaks

\begin{document}
\title{Self-similar solutions of decaying Keller-Segel systems for several populations}
\author{Debabrata Karmakar and Gershon Wolansky}
\thanks{Technion, Israel Institute of Technology, 32000 Haifa, Israel; \\
e-mail: dkarmaker@technion.ac.il, gershonw@math.technion.ac.il }

\begin{abstract}
	It is known that solutions of the parabolic elliptic Keller-Segel equations in the  two dimensional plane
	decay, as time goes to infinity, provided the initial data admits sub-critical  mass and finite second moments, 
	while such solution concentrate, as $t\rightarrow\infty$, in the critical mass.  In the sub-critical case this 
	decay can be resolved by a steady, self-similar solution,
	while no such self similar solution is known to exist for the concentration in the critical case.

This paper is motivated by the Keller-Segel system of several interacting populations, under the existence of an additional 
drift for each component which decays in time at the rate  $O(1/\sqrt{t})$.  We show that self-similar solutions always
exists   in the sub-critical case, while the existence of such self-similar solution in the critical case depends on the 
gap between the decaying drifts for each of the components. For this, we study the conditions for existence/non existence 
of solutions for the corresponding Liouville's systems, which, in turn, is related to the existence/non existence  
of minimizers to a corresponding Free Energy functional. 
\end{abstract}

\maketitle


\section{Introduction}  
The Keller-Segel system represents the   evolution of living cells under self-attraction and diffusive forces \cite{KS}, 
\cite{Pa}. Its general form is given by
\be\label{KS1}\frac{\partial\rho}{\partial t} = \Delta \rho-\nabla\cdot\rho\left(a \nabla_xu- \vec{V}\right) \ ; \ \ 
(x,t)\in \rt\times \mathbb{R}_+\ee
where $a>0$, $\vec{V}=\vec{V}(x,t)$ is a  given vector-field, $\rho=\rho(x,t)$ stands for the distribution of living cells 
and $u=u(x,t)$ the concentration of the chemical substance attracting the cells. In the parabolic/elliptic limit this 
concentration is given by the Newtonian potential
\be\label{KS2}u(x,t):=- \frac{1}{2\pi}\int_{\R^2}\rho(y,t)\ln|x-y|d^2y \ , \ \ i.e\ \  -\Delta u=\rho \ .  \ee 
Since \eqref{KS1} is a parabolic equation of divergence type it follows that the {\em total population number} 
$\int\rho d^2x:=\beta>0$ is conserved in time under suitable boundary conditions at infinity. If $\vec{V}\equiv 0$ then the steady 
states of (\ref{KS1},\ref{KS2}) takes the form of {\em Liouville's Equation}
\be\label{lio}\Delta u(x) + \frac{\beta e^{au(x)}}{\int_{\rt} e^{au(z)}d^2z}=0 \ .  \ee
The spacial dimension 2 which we discuss here was studied by many authors in the case 
$V\equiv 0$ \cite{BD, BT,BG6,BMC}. The two dimensional case is special in the sense that there is a 
critical mass $\beta_c=8\pi/a$. If $\beta<\beta_c$ then, under some natural assumptions on the initial 
data $\rho(x,0):=\rho_0$, the solutions exists globally in time and, moreover, $\lim_{t\rightarrow\infty}\rho(x,t)=0$ 
locally uniformly on $\rt$ \cite{BD}. In particular, there is no solution of \eqref{lio}. If $\beta>\beta_c$ then there is no global 
in time solution of (\ref{KS1}, \ref{KS2}) \cite{HV}  and, again, no solution of \eqref{lio} exists. In the case $\beta=\beta_c$ there 
is a family of solutions of \eqref{lio} and the (free-energy) solutions of (\ref{KS1}, \ref{KS2}) exist globally in time. Moreover,
if the initial data has finite second moment then any such 
solution converges asymptotically to the Dirac measure $\beta_c\delta_0$ \cite{BMC}, otherwise, any radial solution
to (\ref{KS1}, \ref{KS2}) converges asymptotically to
one of the solutions of \eqref{lio} \cite{BG6}.
\par
In the sub-critical case $\beta\leq \beta_c$ it is natural to ask whether there exists self similar solutions of
(\ref{KS1},\ref{KS2}) of the form 
\be\label{rodef} \rho(x,t):= (2t)^{-1}\brho\left(\frac{x}{\sqrt{2t}}, \frac{1}{2}\ln 2t\right)\ , \ \ 
u(x,t)=\bu\left(\frac{x}{\sqrt{2t}},
\frac{1}{2}\ln 2t\right) \ .  \ee
where $t>0$. 
 
It follows that
$$ \bu(y,t)=- \frac{1}{2\pi}\int_{\R^2}\brho(x,t)\ln|x-y|d^2x -\frac{\beta}{2\pi}t$$
in particular $\nabla_x u(x,t)=(2t)^{-1/2}\nabla_y \bu(x/\sqrt{2t}, \frac{1}{2}\ln 2t)$. Substituting in the KS equation we get under
the change of variables $x\rightarrow \frac{x}{\sqrt{2t}}$, $t\rightarrow \frac{1}{2}\ln 2t$,
\be\label{mKS}\partial_t\brho = \Delta\brho- \nabla\cdot \brho\left(a\nabla\bu-x \right)  , \ee
which corresponds to \eqref{KS1} under $\vec{V}(x,t)=x$. 
The corresponding steady state of \eqref{mKS} is 
\be\label{liom}
\Delta_x \bu + \frac{\beta e^{a\bu-|x|^2/2}}{\int_{\rt} e^{a\bu(z)-|z|^2/2}d^2z}=0\ee
The existence and uniqueness (up to a constant) of the solutions to \eqref{liom} in the  sub-critical case $\beta<\beta_c$ was 
given in \cite{CLM,CK94}. 
In \cite{BDP} the authors considered the existence of such self-similar solution of \eqref{rodef} for sub-critical  data. 
Non existence  of solutions of  \eqref{liom} in the critical case was also proved in \cite{CK94}.  
 \par
In this paper we consider \eqref{KS1} with a non-zero, but decaying in time vector field. In particular we assume 
\be\label{Vdef} \vec{V}(x,t)=-v(2t)^{-1/2}\ee
 where $v\in\rt$ is a constant vector. Then  we get under the scaling \eqref{rodef} the following 
modification of \eqref{mKS}:
\be\label{mKSv}\partial_t \brho = \Delta\brho-\nabla\cdot \brho\left(\beta\nabla\bu-(x-v) \right)  \ , \ee
and the modified Liouville's equation
\be\label{liomv}
\Delta_x \bu + \frac{\beta e^{a\bu-|x-v|^2/2}}{\int_{\rt} e^{a\bu(z)-|z-v|^2/2}d^2z}=0 \ . \ee
Evidently, any solution of \eqref{liom} is transformed into a solution of \eqref{liomv} by a shift $x\rightarrow x+v$ and 
v.v. In particular, the self similar solutions of (\ref{KS1}, \ref{KS2}) in the case $\vec{V}=0$ is translated to the 
case of $\vec{V}=-v(2t)^{-1/2}$ by this shift. Thus the non-existence of  global, self-similar solutions of the form 
\eqref{rodef} in the case of critical mass $\beta=\beta_c$  \cite{NS04} is obtained under (\ref{Vdef}) as well.
\par
 In this paper we are motivated by  a generalization of (\ref{KS1}, \ref{KS2}) to the case of a system of $n$ populations
\be\label{KS1n}\frac{\partial\rho_i}{\partial t} = \Delta \rho_i-\nabla\cdot\rho_i\left(\sum_{j=1}^na_{ij} 
\nabla_xu_j- \vec{V}_i\right) \ ; \ \ (x,t)\in \rt\times \mathbb{R}_+\ee
where $A:=(a_{ij})_{n\times n}$ is a symmetric and nonnegative (i.e., $a_{ij}\geq 0$ for all $i,j$) matrix and
\be\label{KS2n}u_i(x,t):=- \frac{1}{2\pi}\int_{\R^2}\rho_i(y,t)\ln|x-y|d^2y \ .  \ee 
\par 
In the case $\vec{V}_i=0$ the stationary solution of such systems, subjected to the initial data satisfying 
$\int\rho_i(x,0)d^2x=\beta_i$ solves the {\em Liouville's systems:}
\be\label{lios} \Delta u_i+ \frac{\beta_ie^{\sum_j a_{ij}u_j}}{\int_{\rt}e^{\sum_j a_{ij}u_j(z)}d^2z}=0 \ . \ee
Again, such Liouville's systems have been studied intensively in \cite{CSW, SW, Lin}, and the cases where
$a_{ij}$ are not necessarily nonnegative (in connection with the chemotactic system known as the conflict
case) have also been explored in \cite{H,W2}.

The solvability of such systems was considered in  \cite{CSW,SW} and \cite{W1}. The criticality condition 
is determined, in that case,  by the functions 
\begin{align*}
\Lambda_J(\bm{\beta}) = \sum_{i \in J}\beta_i \left(8\pi - \sum_{j \in J}a_{ij}\beta_j\right).
\end{align*}
where $\phi \neq J\subseteq I:=\{1, \ldots n\}$. The criticality condition $\beta_c=8\pi/a$ in the case of single composition is 
replaced by 
$$ \Lambda_I(\bm{\beta})=0\ . $$
In particular 
 it was proved in \cite{CSW}  that an entire solution of \eqref{lios} exists {\em only} in the critical case  iff, in addition, 
 $\Lambda_J(\bm{\beta}) >0$ for all $\phi \neq J \subsetneq I$ hold.
 
 In this paper we consider the implementation of 
\be\label{vi}\vec{V}_i=-\frac{1}{\sqrt{2}}t^{-1/2}v_i\ee
in (\ref{KS1n}, \ref{KS2n}), where $v_i\in\rt$ are (perhaps different) constant vectors.  Under the scaling 
\eqref{rodef} we recover  the   modified KS system
\be\label{MKS1n}\frac{\partial\bar\rho_i}{\partial t} = \Delta \bar\rho_i-\nabla\cdot\bar\rho_i\left(\sum_{j=1}^na_{ij} 
\nabla_x\bar u_j- (x-v_i)\right) \ ; \ \ (x,t)\in \rt\times \mathbb{R}_+\ee
where 
\be\label{MKS2n}\bar{u}_i(x,t):=- \frac{1}{2\pi}\int_{\R^2}\bar \rho_i(y,t)\ln|x-y|d^2y -\frac{\beta_i}{2\pi}t \ .  \ee  
The steady states of (\ref{MKS1n}, \ref{MKS2n}) are given by the 
modified  Liouville's system  
\be\label{liosm} \Delta_x \bar{u}_i+ 
\frac{\beta_ie^{\sum_j a_{ij}\bar{u}_j-|x-v_i|^2/2}}{\int_{\rt}e^{\sum_j a_{ij}\bar{u}_j(z)-|z-v_i|^2/2}d^2z}=0. \  \ee
The existence of entire solutions to \eqref{liosm} is, thus, directly related to the existence of self-similar solutions
of the form \eqref{rodef} for (\ref{KS1n}, \ref{KS2n}) under \eqref{vi}. 

The modified KS system (\ref{MKS1n}, \ref{MKS2n}) and the modified  Liouville's system  \eqref{liosm} are closely related
to the {\em Free energy functional}

\begin{multline}\label{Fv}
\mathcal{F}_{\bm{v}}(\bm{\brho}) :=
\sum_{i=1}^n  \int_{\rt} \brho_i(x) \ln \brho_i(x)d^2x + \frac{1}{4\pi}
\sum_{i=1}^n \sum_{j=1}^n a_{ij} \int_{\rt} \int_{\rt} \brho_i(x)\ln|x-y|\brho_j(y)d^2xd^2y \\
+\sum_{i=1}^n \frac{1}{2} \int_{\rt} |x-v_i|^2 \brho_i(x)d^2x,
\end{multline}
defined over the set
\begin{align*}
\Gamma^{\bm{\beta}} := \left\{\bm{\brho} = (\brho_1, \cdots, \brho_n) | \ \brho_i \geq 0,
\int_{\mathbb{R}^2} \brho_i \ln \brho_i <\infty,
\int_{\mathbb{R}^2} |x|^2 \brho_i < \infty, \int_{\mathbb{R}^2} \brho_i = \beta_i, \forall i\right\}.
\end{align*}
Indeed, we observe formally that (\ref{MKS1n}, \ref{MKS2n}) can be written as  a gradient descend system in the Wasserstein 
sense \cite{AGS}
\be\label{Was} \frac{\partial\brho_i}{\partial t} 
= \nabla\cdot\left(\brho_i\nabla\left(\frac{\delta\mathcal{F}_{\bm{v}}}{\delta\brho_i}\right)\right) \ , i=1,\ldots, n, \ee
and, in particular
\be \frac{d}{dt} \mathcal{F}_{\bm{v}}(\bm{\brho})
=-\sum_i\int_{\rt} \rho_i\left| \nabla \frac{\delta\mathcal{F}_{\bm{v}}}{\delta\brho_i}\right|^2.\ee

Every critical point of  $\mathcal{F}_{\bm{v}}$ on $ \Gamma^{\bm{\beta}}$ 
induces a solution  of 
\eqref{liosm}  \cite{CSW}, \cite{Su}.  In particular, any minimizer is such a solution. Moreover, we expect such 
minimizers to be a stable stationary  solutions of (\ref{MKS1n}, \ref{MKS2n}) and thus to represent stable self 
similar limit of (\ref{KS1n}, \ref{KS2n}) under \eqref{vi}. 

Let
\be\label{var}Var(v_1, \ldots, v_n):=\min_{x\in\rt} \sum_{i\in I} |v_i-x|^2. \  \ee
Unless otherwise stated, in this article we assume the matrix $A=(a_{ij})_{n \times n}$ satisfies
\begin{align*}
 (H) \ \ \ \ \ A \ is \ symmetric \ and \ nonnegative,
\end{align*}
and $\bm{\beta}$ satisfies
 \begin{align} \label{beta condition}
 \begin{cases}
 \Lambda_J(\bm{\beta}) \geq 0, \ \mbox{for all} \ \emptyset \neq J \subseteq I, \\
 \mbox{if, \  for some $J\not=\emptyset$}  \ , \ \Lambda_J(\bm{\beta}) = 0, \ \mbox{then} \ a_{ii} +
 \Lambda_{J\backslash \{i\}} > 0,
 \forall i \in J.
 \end{cases}
 \end{align}

The main result of this article is: 
\begin{thm}\label{main}
Suppose $A$ satisfies $(H)$ and $\bm{\beta}$ satisfies
\eqref{beta condition}.
Then 
\begin{itemize}
\item[(a)] $\mathcal{F}_{\bm{v}}$ is bounded from below on $\Gamma^{\bm{\beta}}.$
\item[(b)] If $\Lambda_J(\bm{\beta}) > 0$ for all $\emptyset \neq J \subseteq I,$ then there exists a minimizer
of $\mathcal{F}_{\bm{v}}$ on $\Gamma^{\bm{\beta}},$ for all $(v_1,\cdots,v_n) \in (\rt)^n.$
\item[(c)] If $\Lambda_I(\bm{\beta}) = 0$ and $Var(v_1, \ldots, v_n)=0$ then there is no minimizer of $\mathcal{F}_{\bm{v}}$ in $\Gamma^{\bm{\beta}}$.
\item[(d)] If $n=2$ and $\Lambda_{\{1,2\}}(\bm{\beta})=0$,  $\Lambda_{\{1\}}(\bm{\beta}),
\Lambda_{\{2\}}(\bm{\beta})>0$
and $|v_1-v_2|$ is large enough then there exists a minimizer of
$\mathcal{F}_{\bm{v}}$ on $\Gamma^{\bm{\beta}}$.
\end{itemize}
\end{thm}

For a given such matrix $A,$ we define
\begin{defi}\label{def1} . 
	\begin{itemize}
	\item $\bm{\beta}$ is sub-critical if $\Lambda_J(\bm{\beta})>0$ for any $\emptyset\not= J \subseteq I$.
	\item $\bm{\beta}$ is critical if $\Lambda_I(\bm{\beta})=0$ and
	$\Lambda_J(\bm{\beta})>0$ for any $\emptyset\not= J \subset I$.
\end{itemize}   
\end{defi}
 \begin{thm} \label{cor}
 \noindent
 \begin{itemize}
 \item[(a)] There exists a solution of \eqref{liosm} for any sub-critical $\bm{\beta}$ and any $v_1, \ldots, v_n\in\rt$. 
  \item[(b)] If $\bm{\beta}$ is critical, $Var(v_1, \ldots, v_n)=0,$ and $A$ is invertible and irreducible, then there is no solution 
  to \eqref{liosm}.  
  \item[(c)] There exists a solution of \eqref{liosm} for $n=2$ in the critical  case provided $|v_1-v_2|$ is large enough.  
 \end{itemize}
  \end{thm}
\begin{rem} . 
	\begin{itemize}
\item 	Theorem \ref{cor}-a,c follows immediately from Theorem \ref{main}-a,b,d. 
\item Theorem \ref{main}-c implies the non-existence of minimizers in the critical case. The non-existence of solutions in the critical case (Theorem \ref{cor}-c) follows from a different argument. 
\item  The results of Theorem \ref{main}-d and Theorem \ref{cor}-c can be easily extended to the case $n>2$, provided $Var(v_1,\ldots, v_n)$ is large enough.  
It is not known if $Var(v_1, \ldots, v_n)\neq 0$ is sufficient for existence  of  solutions of (\ref{liosm}) in the critical case for any $n\geq 2$. 
\end{itemize}
 \end{rem}
 


Our organization of the article is as follows: in Section $2$ we discussed the boundedness from below of the functional 
$\mathcal{F}_{\bm{v}}$ over $\Gamma^{\bm{\beta}}.$ Section $3$ is devoted to the basic lemmas required for the proof of our
main theorem. In Section $4$ we proved the existence of minimizers for sub critical $\bm{\beta}.$ The critical case has been
analyzed in Sections $5$ and $6$
and we established an if and only if criterion (Proposition \ref{alternatives}) for the existence of 
minimizers. More precisely, we proved that either a minimizer exists or equality holds in \eqref{funct ineq}. 
At the end of this article we exhibited certain examples (when $Var(v_1, v_2)$ large) for which the minimum is 
actually attained and proved the nonexistence result (Theorem 2(b)) when $Var(v_1,\ldots,v_n)=0.$

\section{Boundedness from below}
 Since we can shift $(v_1, \ldots, v_n)$ by any constant vector we can set   $v_1=v_2=\ldots =v_n=0$ if $Var(v_1, \ldots, v_n)= 0$.  The functional 
  $\mathcal{F}_{\bm{v}}$ will be denoted by $\mathcal{F}_{0}$ in that case. Also, we omit the bars from $\brho_i$ from now on.

  We will actually prove the boundedness from below of a little more general functional. 
  For $\bm{\alpha}:=(\alpha_1,\cdots,\alpha_n) \in (\mathbb{R}_+)^n,$ (where $\mathbb{R}_+$ is the set of all positive real numbers) define
  \begin{align}\label{Falphav}
   \mathcal{F}_{\bm{v},\bm{\alpha}}(\bm{\rho}) :=
\sum_{i=1}^n  \int_{\rt} \rho_i(x) \ln \rho_i(x)d^2x + \frac{1}{4\pi}
\sum_{i=1}^n \sum_{j=1}^n a_{ij}& \int_{\rt} \int_{\rt} \rho_i(x)\ln|x-y|\rho_j(y)d^2xd^2y \notag\\
&+\sum_{i=1}^n \alpha_i \int_{\rt} |x-v_i|^2 \rho_i(x)d^2x.
\end{align}
When $v_i = 0$ for all $i,$ it will be denoted by $\mathcal{F}_{0,\bm{\alpha}}.$

  \begin{thm} \label{bounded from below}
 Condition \eqref{beta condition} is necessary and sufficient condition for the boundedness from below
 of $\mathcal{F}_{\bm{v},\bm{\alpha}}$ on $\Gamma^{\bm{\beta}}.$
\end{thm}

\begin{proof}
	First we recall \cite{CSW,SW} that if $\bm{\rho}$ is supported in a given  bounded set 
	then  $\mathcal{F}_{\bm{v},\bm{0}}$ is bounded from below iff \eqref{beta condition} is satisfied. This implies the necessary part. For the sufficient part we know from the same references that \eqref{beta condition} together with the condition $\Lambda_I(\bm{\beta})=0$ imply that $\mathcal{F}_{\bm{v},\bm{0}}$
	is bounded from below. We only need to show that for any positive $\bm{\alpha}$ we still obtain the bound from below in the case $\Lambda_I(\bm{\beta})>0$. Note also that since $|x-v|^2>|x|^2/2-C$ for any $x\in \R^2$ and $C$ depending on $|v|$ it is enough to prove the sufficient condition for $\bm{v}=0$.

The proof is a straight forward adaptation of the corresponding proof in \cite{SW} without the potential $|x|^2.$
For $\bm{\rho} = (\rho_1,\cdots, \rho_n) \in \Gamma^{\bm{\beta}}$ let $\rho_i^*$ be
the symmetric decreasing rearrangement of $\rho_i.$ Then clearly we have
\begin{align*}
 \int_{\rt} \rho_i \ln \rho_i &= \int_{\rt} \rho_i^* \ln \rho_i^*, \ \int_{\rt} \rho_i |\ln \rho_i|
 = \int_{\rt} \rho_i^* |\ln \rho_i^*|, \
 \int_{\rt} |x|^2 \rho_i^* \leq \int_{\rt} |x|^2\rho_i.
\end{align*}
Thus if we define $\bm{\rho}^* = (\rho_1^*, \cdots, \rho_n^*)$ then $\bm{\rho}^* \in \Gamma^{\bm{\beta}}.$
Furthermore, we have (see \cite{CLoss,SW})
\begin{align*}
 \int_{\rt} \rho_i^*(x) \ln |x-y| \rho_j^*(y) \leq \int_{\rt} \rho_i(x) \ln |x-y| \rho_j(y), \ \forall i,j.
\end{align*}
and hence $\mathcal{F}_{0,\bm{\alpha}} (\bm{\rho}^*) \leq \mathcal{F}_{0,\bm{\alpha}} (\bm{\rho}).$ Therefore
it is enough to prove the theorem for radially symmetric decreasing function of $|x|.$
Let $\bm{\rho} \in \Gamma^{\bm{\beta}}$ be a radially symmetric decreasing function of $r=|x|.$
As in \cite{CSW,SW} we define
\begin{align*}
 m_i(r) &= 2\pi\int_0^r \tau \rho_i(\tau) \ d\tau, \ r \in (0,\infty), \\
 u_i(x) &= -\frac{1}{2\pi}\int_{\rt} \ln|x-y|\rho_i(y) \ d^2y.
\end{align*}
Then we get using $\int_{\rt}|x|^2\rho_i < \infty$ and \cite[equation $(5.6)$]{SW}
\begin{align} \label{asymp}
\begin{cases}
 \lim_{R\rightarrow \infty} \left[u_i(R) + \frac{\beta_i}{2\pi} \ln R\right] = 0, \\
 \lim_{R\rightarrow \infty} (\beta_i - m_i(R))R^2 = 0.
 \end{cases}
 \end{align}
Furthermore, if we define
\begin{align*}
\mathcal{F}_{0,\bm{\alpha},R}(\bm{\rho}) :=
 \sum_{i=1}^n  \int_{B(0,R)} \rho_i \ln \rho_i - \frac{1}{2}
 \sum_i \sum_j a_{ij} \int_{B(0,R)}  \rho_i u_j
  +\sum_{i=1}^n \alpha_i \int_{B(0,R)} |x|^2 \rho_i(x),
 \end{align*}
then by dominated convergence theorem we have
\begin{align*}
 {\mathcal{F}_{0,\bm{\alpha}}}(\bm{\rho}) = \lim_{R \rightarrow \infty}\mathcal{F}_{0,\bm{\alpha},R}(\bm{\rho}).
\end{align*}
Again following \cite{SW}, we define $w_i(s) = m_i(e^s).$ Then  
\begin{align*}
 \sum_{i=1}^n\alpha_i  \int_{B(0,R)} |x|^2 \rho_i(x) \ d^2x &= \sum_{i=1}^n 2\pi\alpha_i \int_0^R r^3 \rho_i(r) \ dr \\
 &= \sum_{i=1}^n \alpha_i \int_0^R r^2 m_i^{\prime}(r) \ dr \\
 &= -2\sum_{i=1}^n \alpha_i \int_0^R rm_i(r) \ dr + \sum_{i=1}^n \alpha_im_i(R)R^2 \\
 &= -2\sum_{i=1}^n \alpha_i \int_{-\infty}^{\ln R} e^{2s}w_i(s) \ ds + \sum_{i=1}^n \alpha_i m_i(R)R^2
\end{align*}
and therefore we can write $\mathcal{F}_{0,\bm{\alpha},R}(\bm{\rho}) = G_R(w) - (\ln 2\pi)\sum_{i=1}^n m_i(R),$ where
\begin{align*}
 G_R(w) &= \int_{-\infty}^{\ln R} \sum_{i=1}^n w_i^{\prime} \ln w_i^{\prime} \ ds+
 \int_{-\infty}^{\ln R} \left[2\sum_{i=1}^n w_i - \frac{1}{4\pi}\sum_{i,j=1}^na_{ij}w_iw_j \right] \ ds\notag \\
 &-2\sum_{i=1}^n \alpha_i\int_{-\infty}^{\ln R} e^{2s}w_i \ ds- \sum_{i=1}^n m_i(R)\left(2\ln R
 + \frac{1}{2}\sum_{j=1}^na_{ij}u_j(R) - \alpha_iR^2\right).
\end{align*}
Now define $\nu_i = 2 - \frac{1}{4\pi}\sum_{j=1}^na_{ij}\beta_j.$ Using the identity $\frac{\Lambda_I(\bm{\beta})}{4\pi}
= \sum_{i=1}^n \nu_i\beta_i$ and \eqref{asymp} we get
\begin{align}\label{Gr}
 -\sum_{i=1}^n m_i(R)\left[2 \ln R + \frac{1}{2} \sum_{j=1}^na_{ij}u_j(R)\right] +
 \sum_{i=1}^n 2\nu_i\beta_i \ln R = \frac{\Lambda_I(\bm{\beta})}{4\pi} \ln R+ o_{R}(1),
\end{align}
where $o_R(1)$ stands for a quantity going to zero as $R\rightarrow \infty.$ Utilizing \eqref{Gr}, we can decompose 
$G_R(w)$ as follows
\begin{align*}
G_R(w) = J_{-\infty}(w) + J_{\infty}(w) + E_R(w) + o_R(1),
\end{align*}
where
\begin{align*}
 J_{-\infty}(w) &= \int_{-\infty}^{0} \sum_{i=1}^n w_i^{\prime} \ln w_i^{\prime} \ ds+
 \int_{-\infty}^{0} \left[2\sum_{i=1}^n w_i - \frac{1}{4\pi}\sum_{i,j=1}^na_{ij}w_iw_j \right] \ ds \\
 &-2\sum_{i=1}^n \alpha_i\int_{-\infty}^0 e^{2s}w_i \ ds, \\
 J_{\infty}(w) & = \int_{0}^{\ln R} \sum_{i=1}^n w_i^{\prime} \ln w_i^{\prime} \ ds \\
 &+ \int_{0}^{\ln R} \left[\sum_{i=1}^n 2(1-\nu_i)w_i - \frac{1}{4\pi}\sum_{i,j=1}^na_{ij}w_iw_j +
 \frac{\Lambda_I(\bm{\beta})}{4\pi}\right] \ ds \\
 E_R(w) &= -2\sum_{i=1}^n \alpha_i \int_{0}^{\ln R} e^{2s}w_i \ ds + \sum_{i=1}^n 2\nu_i\int_0^{\ln R} w_i \ ds
 -2\left(\sum_{i=1}^n\nu_i\beta_i\right)\ln R \\
 &+\sum_{i=1}^n \alpha_i m_i(R)R^2.
 \end{align*}

By \cite{SW} we have $J_{-\infty}$ and $ J_{\infty}$ are bounded from below on $\Gamma^{\bm{\beta}},$ once we observe that
\begin{align*}
 \int_{-\infty}^0e^{2s}w_i \leq \beta_i\int_{-\infty}^0e^{2s} \leq \frac{\beta_i}{2}.
\end{align*}
Therefore, we only need to show that $E_R(w)$ is bounded from below. We can rewrite $E_R(w)$ in the following way
\begin{align*}
 E_R(w) &= \int_{0}^{\ln R} \left[\sum_{i=1}^n 2\left(\nu_i - \alpha_ie^{2s}\right)w_i -2\sum_{i=1}^n \nu_i \beta_i
 + 2\sum_{i=1}^n \alpha_i \beta_i e^{2s} \right] \ ds \\
 &= \int_{0}^{\ln R} \left[2\sum_{i=1}^n (\beta_i - w_i(s))(\alpha_ie^{2s} - \nu_i) \right] \ ds.
\end{align*}

Now $w_i(s) \leq \beta_i$ for all $s$ and $\alpha_i>0, \nu_i$ are being fixed numbers, we can find a $R_0>0,$ independent of
$w_i$ such that $(\beta_i - w_i(s))(\alpha_ie^{2s} - \nu_i) \geq 0$ for all $s \geq \ln R_0.$ Again since
\begin{align*}
 \left|\int_0^{\ln R_0} \left[2\sum_{i=1}^n (\beta_i - w_i(s))(\alpha_ie^{2s} - \nu_i) \right] \ ds\right|
 \leq \sum_{i=1}^n 4\beta_i \left(\frac{\alpha_i}{2}R_0^2 - \nu_i\ln R_0 - \frac{\alpha_i}{2}\right).
\end{align*}
we have $E_R(w) \geq -|E_{R_0}(w)| \geq -C.$ This proves the sufficiency of the condition \eqref{beta condition}.
\end{proof}

\section{Basic Lemmas}
In this section we will recall a few definitions and lemmas and also prove some basic ingredients required for 
the proof of our main results. We define the space $\mathbb{L}\ln \mathbb{L}(\rt)$ as the Orlicz space 
determined by the $N$-function $N(t) = (1 + t)\ln (1+t) - t, t \geq 0$:
\begin{align*}
 \mathbb{L}\ln \mathbb{L}(\rt) := \left\{\rho : \rt \rightarrow \mathbb{R} \ \mbox{measurable} \ : 
 \int_{\rt} [(1 + |\rho|)\ln (1 + |\rho|) - |\rho|] d^2x < \infty \right\}.
\end{align*}
Then $\mathbb{L}\ln \mathbb{L}(\rt)$ is a Banach space with respect to the Luxemberg norm (because 
$N(t)$ satisfies the $\Delta_2$ condition: $N(2t) \leq 2N(t)$ for all $t \geq 0$).

The dual space of $\mathbb{L}\ln \mathbb{L}(\rt)$ is the Orlicz space determined by the 
$N$-function $M(t) = (e^t-t-1), t \geq 0.$ It is important to remark that $\mathbb{L}\ln \mathbb{L}(\rt)$ is not reflexive
(because $M(t)$ does not satisfy the $\Delta_2$ condition). However, there is a notion of weak convergence
which is slightly weaker than the usual weak convergence in Banach spaces. 
A sequence $\rho_m \in \mathbb{L}\ln \mathbb{L}(\rt)$
is said to converge $L_M$-weakly to $\rho$ if
\begin{align*}
 \int_{\rt} \rho_m \phi \rightarrow \int_{\rt} \rho \phi, \ \mbox{for all bounded measurable functions $\phi$ with bounded 
 support}. 
\end{align*}
It is well known   from the general Orlicz space theory \cite{Kr} that $\mathbb{L}\ln \mathbb{L}(\rt)$ is $L_M$-weakly compact. To
simplify our notations we will denote the weak convergence (in the above sense) by $\rho_m\rightharpoonup \rho.$ 

We begin with the following elementary lemma whose proof can be found in \cite{BDP}:
\begin{lem} \label{boundedness}
 For $ 1 \leq i \leq n$ let $\rho_i \in L^1(\rt)$ be such that $\rho_i \geq 0$ and satisfies
 \begin{align*}
  \int_{\rt} \rho_i \ln \rho_i \leq C_0,  \int_{\rt}|x|^2\rho_i \leq C_0.
 \end{align*}
Then
\begin{align*}
 \sum_{i=1}^n\int_{\rt} \rho_i |\ln \rho_i| \leq \sum_{i=1}^n\int_{\rt} \rho_i \ln \rho_i +
 2\ln 2\pi\left(\sum_{i=1}^n \int_{\rt} \rho_i\right)  +  2\sum_{i=1}^n \int_{\rt}|x|^2\rho_i + 2ne^{-1}.
\end{align*}
\end{lem}

%

\begin{lem} \label{Fatous lemma}
  Let $\{\rho_m\}$ be a sequence in $\mathbb{L}\ln\mathbb{L}(\rt)$ such that
 \begin{align*}
  \int_{\rt} \rho_m \ln \rho_m \leq C_0, \ \int_{\rt} \rho_m = \beta, \ \int_{\rt} |x|^2\rho_m \leq C_0.
 \end{align*}
 Then there exists $\rho \in \mathbb{L}\ln\mathbb{L}(\rt)$ such that up to a subsequence  $\rho_m \rightharpoonup \rho$ in
 the topology of $\mathbb{L}\ln\mathbb{L}(\rt)$ and satisfies
\begin{align} \label{Fl}
 \int_{\rt} \rho \ln \rho \leq \liminf_{n\rightarrow\infty} \int_{\rt} \rho_m \ln \rho_m.
\end{align}
\end{lem}

\begin{rem}
 The conclusion of the lemma is false without the assumption on the uniform boundedness of $\int_{\rt} |x|^2\rho_m.$ 
 As a counter example, let
 $\phi \in C_c^{\infty}(\rt)$ be a smooth cutoff function such that $0 \leq \phi \leq 1 - \delta,$ for some $\delta \in (0,1).$  
 Let $x_m$ be a sequence in $\rt$ such that $|x_m| \nearrow \infty$ and define the sequence
 \begin{align*}
  \rho_m(x) = \phi(x + x_m).
 \end{align*}
Then it is easy to check that $\int_{\rt} |x|^2\rho_m \rightarrow \infty,$ and
\begin{align*}
 \int_{\rt} \rho_m \ln \rho_m = \int_{\rt} \phi \ln \phi < 0, \ \mbox{for all} \ m.
\end{align*}
But $\rho_m \rightharpoonup \rho \equiv 0$ in $\mathbb{L}\ln\mathbb{L} (\rt)$ and hence $\int_{\rt} \rho \ln \rho = 0.$ Therefore the assumption
$\int_{\rt} |x|^2\rho_m$ bounded is a necessary condition for the Fatou's type Lemma \eqref{Fl} to hold true.
\end{rem}

We need some supplementary lemmas to prove Lemma \ref{Fatous lemma}.
\begin{lem} \label{L1 weak convergence}
 Let $\{\rho_m\}$ be a sequence in $\mathbb{L}\ln\mathbb{L}(\rt)$ such that
 \begin{align} \label{upper bound}
  \int_{\rt} \rho_m \ln \rho_m \leq C_0, \ \int_{\rt} \rho_m = \beta, \ \int_{\rt} |x|^2\rho_m \leq C_0.
 \end{align}
Then there exists a $\rho \in L^1(\rt, (1+|x|^2)d^2x)$ such that (up to a subsequence) $\rho_m\rightharpoonup \rho$ weakly in $L^1(\rt),$
i.e.,
\begin{align*}
 \int_{\rt}\rho_m g \rightarrow \int_{\rt}\rho g, \ \mbox{for all} \ g \in L^{\infty}(\rt).
\end{align*}
\end{lem}

\begin{proof}
By Lemma \ref{boundedness}, the assumption  \eqref{upper bound} implies that
\begin{align*}
 \int_{\rt} |\rho_m| |\ln \rho_m| \leq C.
\end{align*}
for some constant $C,$ and hence $\int_{\rt}[(1 + \rho_m) \ln(1 + \rho_m) -\rho_m ]$ is uniformly bounded.
Since $\int_{\rt} \rho_m = \beta$ by weak* compactness in $L^1$ there exists a finite measure $\mu$ on
$\rt$ such that
\begin{align*}
 \int_{\rt} \rho_m \phi \rightarrow \int_{\rt} \phi \ d\mu, \ \mbox{for all} \ \phi \in C_0(\rt).
\end{align*}
Furthermore, the uniform boundedness of $\int_{\rt}[(1 + \rho_m) \ln(1 + \rho_m) -\rho_m ]$ implies  $\mu$ has a
density $\rho \in L^1_{loc}(\rt).$
Now we claim that $\int_{\rt}|x|^2\rho < +\infty.$ To prove it we let $\phi\in C_0(\mathbb{R}^2)$ such that
$\phi(x) = |x|^2$ in $B(0,R), 0 \leq \phi \leq |x|^2$ in $\rt$. Then by \eqref{upper bound}
and $L^1$ weak* convergence we get
\begin{align*}
 \int_{\{|x| < R\}} |x|^2\rho \leq \int_{\mathbb{R}^2} \rho\phi = \lim_{m\rightarrow\infty} \int_{\mathbb{R}^2} \rho_m\phi \leq C_0.
\end{align*}
Letting $R \rightarrow \infty$ we reach at the desired claim. Moreover, the assumption $\int_{\rt} |x|^2\rho_m \leq C_0$
gives $\int_{\rt} \rho = \beta.$ Therefore, by Portmanteau's theorem
\begin{align} \label{bcf}
 \int_{\rt} \rho_m \phi \rightarrow \int_{\rt} \rho\phi,
\end{align}
for all bounded continuous functions
$\phi$ on $\rt.$ Using Lusin's theorem and Tietz's extension theorem we can extend this result to 
$\phi\in \mathbb{L}^\infty(\mathbb{R}^2)$. 
\end{proof}

\begin{lem} \label{weakly closed}
 The set
 \begin{align*}
  \mathcal{S} := \left\{\rho \in L^1(\rt) : \rho \geq 0, \int_{\rt}\rho \ln \rho \leq \alpha, \int_{\rt} \rho = \beta,
  \int_{\rt}|x|^2\rho \leq C_0\right\}
 \end{align*}
is a weakly closed subset in $L^1(\rt).$
\end{lem}
\begin{proof}
 We will show that the set $\mathcal{S}$ is a convex and strongly closed subset of $L^1(\rt).$ Then by Mazur's lemma it will
imply the weakly closeness of $\mathcal{S}.$ Again by the convexity of $t\ln t$ we only need to show that $\mathcal{S}$ is strongly closed in
$L^1(\rt).$ Let $\{\rho_m\}_m$ be a sequence in $L^1(\rt)$ such that $\rho_m \rightarrow \rho$ in $L^1(\rt).$ Let $\rho_m^{*},
\rho^*$ be the symmetric decreasing rearrangement of $\rho_m$ and $\rho$ respectively. Then $\rho_m^* \rightarrow \rho^*$
in $L^1(\rt)$ and up to a subsequence $\rho_m$ (respectively $\rho_m^*$) converges point wise a.e. in $\rt.$ By strong convergence
 and Fatou's lemma we have
\begin{align*}
 \int_{\rt} \rho = \beta, \ \int_{\rt} |x|^2\rho \leq C_0.
\end{align*}
Furthermore, by Lemma \ref{boundedness} and the point wise convergence we obtain
\begin{align*}
 \int_{\rt} \rho |\ln \rho| < + \infty.
\end{align*}
To conclude the proof of the lemma we will show $\int_{\rt} \rho^*\ln \rho^* \leq \alpha.$ Using Fatou's lemma we get
\begin{align*}
 \int_{B(0,R)} \rho^*\ln \rho^*  \leq \liminf \int_{B(0,R)} \rho_m^*\ln \rho_m^*,
\end{align*}
for any $R>0.$ Now to estimate for $|x|>R$ we will use the bound $0 \leq \rho^*(|x|) \leq \frac{\beta}{\pi|x|^2}.$ The bound
follows from
\begin{align*}
 \beta = \int_{\rt} \rho  = \int_{\rt}\rho^* = 2\pi \int_0^{\infty} s \rho^*(s) ds \geq 2\pi \int_0^r s \rho^*(s) ds \geq
\pi r^2\rho^*(r).
 \end{align*}
Choosing $\epsilon \in (0,\frac{1}{2})$ and using
$\ln (1/t) \leq 1/t$ 
for $t <1$ we get, after multiplying by  $\epsilon$ and using $\rho^*(x) < 1$
 for sufficiently large $R$
\begin{align*}
 \int_{\{|x|>R\}} \rho_m^* |\ln \rho_m^*| &\leq \frac{1}{\epsilon} \int_{\{|x|>R\}} \rho_m^* \frac{1}{(\rho_m^*)^{\epsilon}} \\
 &=\frac{1}{\epsilon} \int_{\{|x|>R\}} \left(\rho_m^*\right)^{1 - \epsilon}, \\
 &= \frac{1}{\epsilon} \int_{\{|x|>R\}}\frac{\left(|x|^2\rho_m^*\right)^{1 - \epsilon}}{|x|^{2(1-\epsilon)}}, \\
 &\leq \frac{1}{\epsilon}\left(\int_{\{|x|>R\}}|x|^2 \rho_m^*\right)^{1-\epsilon}\left(\int_{\{|x|>R\}} |x|^{2(1- \frac{1}{\epsilon})}\right)^{\epsilon},\\
 &= O\left(\frac{1}{R^{2(\frac{1}{\epsilon}-2)}}\right)
\end{align*}
Thus we obtain
\begin{align*}
 \int_{B(0,R)} \rho_m^* \ln \rho_m^* \leq \int_{\rt} \rho_m^* \ln \rho_m^* + O\left(\frac{1}{R^{2(\frac{1}{\epsilon}-2)}}\right),
\end{align*}
and hence
\begin{align*}
 \int_{B(0,R)} \rho^* \ln \rho^* \leq \liminf \int_{\rt} \rho_m^* \ln
 \rho_m^* + O\left(\frac{1}{R^{2(\frac{1}{\epsilon}-2)}}\right).
\end{align*}
Letting $R \rightarrow \infty$ we get the desired result.
\end{proof}

\textbf{Proof of Lemma \ref{Fatous lemma}:}
\begin{proof}
 Define $\alpha = \liminf \int_{\rt} \rho_m \ln \rho_m + \epsilon,$ where $\epsilon>0$ is a small fixed number. Let
 $\rho_{m_k}$ be a subsequence such that $\lim \int_{\rt} \rho_{m_k} \ln \rho_{m_k}=\liminf \int_{\rt} \rho_m \ln \rho_m.$
 By Lemma \ref{L1 weak convergence}, up to a subsequence $\rho_{m_k}$ converges to some $\rho$ weakly in $L^1(\rt).$
 Since for sufficiently large $k, \rho_{m_k} \in \mathcal{S},$ which is weak $L^1$-closed by Lemma \ref{weakly closed},
 we conclude that $\rho\in \mathcal{S}$ and hence
 \begin{align*}
  \int_{\rt} \rho \ln \rho\leq \liminf \int_{\rt} \rho_m \ln \rho_m + \epsilon.
 \end{align*}
Since $\epsilon>0$ is arbitrary the proof of the lemma is completed.
\end{proof}

\begin{lem} \label{asymptotic behavior}
 Let $\rho \in L^1(\rt)$ satisfies
 \begin{align*}
  \int_{\rt} \rho \ln \rho \leq C_0, \int_{\rt} \rho = \beta, \ \int_{\rt} |x|^2\rho \leq C_0.
 \end{align*}
Define
\begin{align}\label{Ulog}
 u(x) = -\frac{1}{2\pi}\int_{\rt} \ln|x-y|\rho(y) \ d^2y, \ \mbox{for} \ x \in \rt.
\end{align}
Then there exists a constants $C,R$ depending only on $C_0$ and $\beta$ such that
\begin{align*}
 \left|u(x) + \frac{\beta}{2\pi}\ln |x|\right| \leq C, \ \mbox{for all} \ |x|>R.
\end{align*}
\end{lem}
\begin{proof}
 The proof goes in the same line as in Chen and Li \cite{CL93} with slight modifications. As in \cite{CL93} we
 write
 \begin{align*}
  \frac{u(x)}{\ln |x|} + \frac{\beta}{2\pi} = -\frac{1}{2\pi} \int_{\rt} \frac{\ln|x-y| - \ln |x|}{\ln|x|}\rho(y) \ d^2y
 = I_1 + I_2 + I_3,
 \end{align*}
where the integral $I_1$ is over the domain $\{|x-y|<1\},I_2$ is over the domain $\{|x-y|>1,|y| \leq \frac{|x|}{2}\}$
and $I_3$ is over the domain $\{|x-y|>1,|y| > \frac{|x|}{2}\}.$ We want to show that each $I_j$ is bounded by
$C(\beta,C_0)/\ln |x|.$ Now
\begin{align} \label{1st int}
 |I_1| \leq \int_{\{|x-y|<1\}} \rho(y) \ d^2y + \frac{1}{\ln|x|}\int_{\{|x-y|<1\}} |\ln|x-y||\rho(y) \ d^2y
\end{align}
Since $\{|x-y|<1\} \subset \{|y|>|x|-1\},$ and $\int_{\rt} |x|^2\rho \leq C_0$ the first integral in \eqref{1st int}
is bounded by $C(\beta,C_0)/(|x|-1)^2.$ To estimate the second integral in \eqref{1st int} we divide it into two parts
$\{|x-y|>1, \rho \leq 1\}$ and $\{|x-y|>1, \rho > 1\}.$ Clearly,
\begin{align*}
 \frac{1}{\ln|x|}\int_{\{|x-y|<1, \rho \leq 1\}}
 |\ln|x-y||\rho(y) \ d^2y \leq \frac{C(\beta,C_0)}{\ln |x|}.
\end{align*}
Choose $\epsilon \in (0,1)$ then
\begin{align*}
\int_{\{|x-y|<1, \rho > 1\}} |\ln|x-y||\rho(y) \ d^2y
 &\leq \int_{\{|x-y|<1, \ln \rho < \epsilon \ln \frac{1}{|x-y|}\}} |\ln|x-y||\rho(y) \ d^2y \\
& +\int_{\{|x-y|<1, \ln \rho > \epsilon \ln \frac{1}{|x-y|}\}} |\ln|x-y||\rho(y) \ d^2y \\
&\leq \int_{\{|x-y|<1\}} \ln \frac{1}{|x-y|} e^{\epsilon \ln \frac{1}{|x-y|}} \\
&+\frac{1}{\epsilon} \int_{\{|x-y|<1\}} \rho \ln \rho \\
&\leq \int_{\{|x-y|<1\}} \ln \frac{1}{|x-y|} \frac{1}{|x-y|^{\epsilon}} \\
&+\frac{1}{\epsilon} \int_{\{|x-y|<1\}} \rho \ln \rho \\
&\leq C(\beta, C_0, \epsilon).
 \end{align*}
Combining all we get the estimate
\begin{align*}
 |I_1| \leq C(\beta, C_0)\left[\frac{1}{\ln |x|} + \frac{1}{(|x|-1)^2}\right].
\end{align*}
To estimate $I_2$ we see that on the domain $\{|x-y|>1,|y| \leq \frac{|x|}{2}\}, |\ln |x-y| - \ln|x|| \leq 1.$  Thus
\begin{align*}
 |I_2| \leq \frac{1}{\ln |x|} \int_{\{|y| \leq \frac{|x|}{2}\}} \rho(y) \ d^2y \leq \frac{C(\beta)}{\ln |x|}.
\end{align*}
Now on $I_3, \ln |x-y| \geq 0, |x-y| \leq 3|y|$ and hence $|\ln|x-y| - \ln |x||\leq \ln3|y|+ \ln |x|.$ Therefore
\begin{align*}
 |I_3| &\leq \frac{1}{\ln|x|} \int_{\{|y| > \frac{|x|}{2}\}} \ln 3 |y|\rho(y) \ d^2y +
 \int_{\{|y| > \frac{|x|}{2}\}}\rho(y) \ d^2y \\
 &\leq C(\beta, C_0) \left[\frac{1}{|x|\ln|x|} + \frac{1}{|x|^2}\right].
\end{align*}
 \end{proof}

 We end this section with the following compactness lemma whose proof can be found in \cite{ST12}.
 \begin{thma} \label{BT}
  Suppose we have a sequence $\{u_m\} \subset H^1(B(0,2R))$ of weak solutions to
  \begin{align} \label{equ1}
   -\Delta u_m = f_m, \ \mbox{in} \ B(0,2R),
  \end{align}
and $\{f_m\} \subset \mathbb{L}\ln\mathbb{L}(B(0,2R)).$ Suppose there exists a constant $C<+\infty$ such that
\begin{align} \label{con1}
 ||u_m||_{L^1(B(0,2R))} + ||f_m||_{\mathbb{L}\ln\mathbb{L} (B(0,2R))} \leq C.
\end{align}
Then there exists $u \in H^1_{\mbox{loc}}(B(0,2R))$ such that
\begin{align*}
 ||u_m - u||_{H^1(B(0,R))} \rightarrow 0, \ \mbox{as} \ m \rightarrow \infty.
\end{align*}
\end{thma}
In \cite{ST12}, the authors actually proved the above compactness theorem for $R=\frac{1}{2}$ but for more general inhomogeneity
$\Omega_m \cdot \nabla u_m + f_m$ under some smallness condition on $\Omega_m.$ For our purpose we can take $\Omega_m \equiv 0,$
and the general $R$ can be dealt with through a simple scaling argument. To be meticulous, define $\tilde u_m(x)=u_m(2Rx)$
and $\tilde f_m(x) = (2R)^2f_m(2Rx).$ Then one can easily verify that \eqref{equ1},\eqref{con1} holds with $u_m,f_m$ replaced by
$\tilde u_m, \tilde f_m$ in the domain $B(0,1).$ Hence by compactness theorem there exists $\tilde u \in H^1_{loc}(B(0,1))$
such that $\tilde u_m \rightarrow \tilde u$ in $H^1(B(0,\frac{1}{2})).$ Scaling back to the original variable we see that 
$u_m(\cdot) \rightarrow u(\cdot) := \tilde u(\frac{\cdot}{2R})$ in $H^1(B(0,R)).$
We refer the reader to \cite{ST12} for more details.

\section{Existence of minimizers: sub-critical case}
In this section we assume $\bm{\beta}$ is sub-critical (Definition \ref{def1}). 

\begin{thm} \label{sub existence}
 If $\bm{\beta}$ is sub-critical  then for all $\bm{v}=(v_1, \cdots, v_n) \in (\rt)^n$ there exists
 a minimizer of $\mathcal{F}_{\bm{v}}$ on $\Gamma^{\bm{\beta}}.$
\end{thm}

\begin{proof}
 Let $\bm{\rho}^m = (\rho_1^m,\cdots, \rho_n^m)$ be a minimizing sequence for $\mathcal{F}_{\bm{v}}$ on $\Gamma^{\bm{\beta}}.$ Since $\bm{\beta}$ is sub-critical  we can choose $\epsilon \in (0,\frac{1}{2})$ small such that
 \begin{align} \label{subcritical epsilon}
  \sum_{i \in J} \beta_i \left(8\pi - \sum_{j \in J}(a_{ij}+\epsilon)\beta_j\right) > 0, \ \mbox{for all} \
  \emptyset \neq J \subset I.
 \end{align}
 Step 1: $\int_{\rt}|x-v_i|^2\rho_i^m$ is uniformly bounded by some constant $C_0.$

By Theorem \ref{bounded from below}
\begin{align*}
 \sum_{i=1}^n\int_{\rt} \rho_i^m \ln \rho_i^m &+ \frac{1}{4\pi} \sum_{i=1}^n\sum_{j=1}^n a_{ij} \int_{\rt}\int_{\rt}
 \rho_i^m(x)\ln |x-y| \rho_j^m(y) \\
 &+ \sum_{i=1}^n (\frac{1}{2} - \epsilon)\int_{\rt}|x-v_i|^2\rho_i^m \geq -C
\end{align*}
which implies $\mathcal{F}_{\bm{v}}(\bm{\rho}^m) - \epsilon\sum_{i=1}^n \int_{\rt}|x-v_i|^2\rho_i^m \geq -C.$ 
Since along a minimizing sequence $\mathcal{F}_{\bm{v}}(\bm{\rho}^m)$ is bounded above, the conclusion of Step $1$
is proved.

Step 2: $\rho_i^m$ are uniformly bounded in $\mathbb{L}\ln\mathbb{L}$.

Define
\begin{align*}
 I_{ij}^m := \int_{\rt}\int_{\rt} \rho_i^m(x)\ln |x-y| \rho_j^m(y).
\end{align*}
Using Step $1$ and the following inequality
\begin{align*}
 \ln|x-y| \leq \frac{1}{2}\ln(1+|x|^2)+ \frac{1}{2}\ln(1+|y|^2) \ \ , \ \ |x|^2 > \ln(1+|x|^2)
\end{align*}
we see that $I_{ij}^m \leq \frac{C_0}{2}(\beta_i + \beta_j).$   
Since $\bm{\beta}$ satisfies \eqref{subcritical epsilon} we obtain by Theorem \ref{bounded from below}
\begin{align*}
\sum_{i=1}^n\int_{\rt} \rho_i^m \ln \rho_i^m &+ \frac{1}{4\pi} \sum_{i=1}^n\sum_{j=1}^n (a_{ij}+\epsilon) \int_{\rt}\int_{\rt}
 \rho_i^m(x)\ln |x-y| \rho_j^m(y) \\
 &+ \sum_{i=1}^n \frac{1}{2}\int_{\rt}|x-v_i|^2\rho_i^m \geq -C.
 \end{align*}
Therefore we have

\begin{align*}
 \mathcal{F}_{\bm{v}}(\bm{\rho}^m) + \frac{\epsilon}{4\pi} \sum_{I_{ij}^m > 0} I_{ij}^m -
 \frac{\epsilon}{4\pi} \sum_{I_{ij}^m < 0} |I_{ij}^m| \geq -C.
\end{align*}
Since along a minimizing sequence $\mathcal{F}_{\bm{v}}(\bm{\rho}^m)$ is bounded we obtain $\sum_{I_{ij}^m < 0} |I_{ij}^m|$
is uniformly bounded. Hence $\sum_{i=1}^n\int_{\rt} \rho_i^m \ln \rho_i^m$ is upper bounded and hence by Lemma \ref{boundedness} we get the uniform bound of $\bm{\rho}^m$ in $\mathbb{L}\ln\mathbb{L}$. 

Step 3: Existence of a limit.

By Lemma \ref{Fatous lemma} there exists $\rho_i \in \mathbb{L}\ln\mathbb{L}(\rt)$ such that 
up to a subsequence $\rho_i^m \rightharpoonup \rho_i$ in the topology of $\mathbb{L}\ln\mathbb{L}$ and satisfies the 
inequality
\begin{align} \label{rho1}
 \int_{\rt} \rho_i \ln \rho_i \leq \liminf \int_{\rt} \rho^m_i \ln \rho^m_i, \ \mbox{for all} \ i.
\end{align}
Furthermore, it also follows from the proof of Lemma \ref{Fatous lemma} that
\begin{align} \label{rho2}
 \sum_{i=1}^n \frac{1}{2}\int_{\rt} |x-v_i|^2\rho_i \leq \liminf \sum_{i=1}^n \frac{1}{2}\int_{\rt} |x-v_i|^2\rho_i^m,
 \ \ \ \int_{\rt} \rho_i = \beta_i,
\end{align}
and hence $\bm{\rho} :=(\rho_1, \cdots, \rho_n) \in \Gamma^{\bm{\beta}}.$ To complete the proof of the theorem we need to 
show that 
\begin{align*}
 \int_{\rt} \rho_i^mu_j^m \rightarrow \int_{\rt} \rho_iu_j \ \mbox{for all} \ 1 \leq i,j \leq n,
\end{align*}
where $u_j, u^m_j$ are defined by (\ref{Ulog}) via $\rho_j$, $\rho^m_j$ respectively. 

By Lemma \ref{asymptotic behavior} we have for $R$ large
\begin{align} \label{1}
 \left|\int_{\{|x|>R\}} \rho_i^mu_j^m\right| &\leq C \left[\int_{\{|x|>R\}} \rho_i^m \ln |x| + \int_{\{|x|>R\}}
 \rho_i^m\right] \\
 &\leq \frac{C}{R}.
\end{align}
For $\{|x|\leq R\}$ we will use Theorem \ref{BT} to prove the convergence. For that we need to show that $u_i^m \in H^1_{loc}(\rt)$
and $||u_i^m||_{L^1(B(0,2R))}$ is uniformly bounded for all $i= 1,\cdots,n:$

\begin{align*}
 \int_{\{|x|<2R\}} |u_i^m| &\leq \frac{1}{2\pi} \int_{\{|x|<2R\}}\int_{\rt}|\ln|x-y||\rho_i^m(y) \ d^2y d^2x \\
 &\leq \int_{\{|x|<2R\}}\int_{\{|y|< 4R\}}|\ln|x-y||\rho_i^m(y) \ d^2y d^2x \\
 &+ \int_{\{|x|<2R\}}\int_{\{|y|>4R\}}|\ln|x-y||\rho_i^m(y) \ d^2y d^2x \\
 & \leq \int_{\{|y|<4R\}}\rho_i^m(y) \int_{\{|x|<2R\}} |\ln|x-y||\  d^2x d^2y + C(R)\int_{\rt} |y|^2\rho_i^m(y) \ d^2y \\
 &\leq C(R).
\end{align*}


By compactness Theorem \ref{BT}, there exists $\tilde{u}_i \in H^1(B(0,R))$ such that
$\tilde{u}_i^m$ converges to $u_i$ in $H^1(B(0,R)).$ Therefore $u_i^m$ converges to $\tilde{u}_i$ in the strong topology of
Orlicz space determined by the $N$-function $(e^t-t-1).$ By duality
\begin{align} \label{2}
 \int_{B(0,R)}\rho_i^m u_j^m \rightarrow \int_{B(0,R)} \rho_iu_j.
\end{align}
Hence by \eqref{1} and \eqref{2} we see that
\begin{align} \label{3}
 \int_{\rt}\rho_i^m u_j^m \rightarrow \int_{\rt} \rho_iu_j, \ \mbox{for all} \ i,j.
\end{align}
Therefore by \eqref{rho1}, \eqref{rho2} and \eqref{3} we have $\bm{\rho} \in \Gamma^{\bm{\beta}}$ and
\begin{align*}
 \mathcal{F}_{\bm{v}}(\bm{\rho}) \leq \liminf \mathcal{F}_{\bm{v}}(\bm{\rho}^m) = \inf_{\Gamma^{\bm{\beta}}} 
 \mathcal{F}_{\bm{v}}.
\end{align*}
This completes the proof of the theorem.
\end{proof}

\begin{rem}
 It follows from the proof of Theorem \ref{sub existence} that if a minimizing sequence is bounded in the 
 $\mathbb{L}\ln\mathbb{L}$
 topology and has bounded second moment then the minimizing sequence converges and the limit is a minimizer. More
 precisely, if $\bm{\rho}^m$ is minimizing sequence that satisfies
 \begin{align*}
  \sum_{i=1}^n\int_{\rt} \rho_i^m\ln \rho_i^m \leq C_0, \ \mbox{and} \ \sum_{i=1}^n\int_{\rt} |x|^2\rho_i^m \leq C_0
 \end{align*}
for some constant $C_0$ independent of $m$ then there exists $\bm{\rho}_0 \in \Gamma^{\bm{\beta}}$ such that
$\bm{\rho}^m \rightharpoonup \bm{\rho}_0$ in the topology of $\mathbb{L}\ln\mathbb{L}$ 
and $\mathcal{F}_{\bm{v}}(\bm{\rho}_0) = \inf_{\bm{\rho} \in \Gamma^{\bm{\beta}}}
\mathcal{F}_{\bm{v}}(\bm{\rho}).$
\end{rem}

\section{The Critical case}

Recall the definition of the functional $\mathcal{F}_{\bm{v}}(\bm{\rho})$  \eqref{Fv}. In this section we assume the critical case
\begin{equation} \label{critical} \Lambda_I(\bm{\beta})=0 \ \ , \Lambda_J(\bm{\beta})>0 \ \forall J\subset I, J\not=I,\emptyset \ . \end{equation}

\begin{lem} \label{concentration}
 If $\bm{\beta}$ satisfies \eqref{critical} then $\mathcal{F}_{0}$ does not attain its infimum
 on $\Gamma^{\bm{\beta}}.$
\end{lem}
 
\begin{proof}
Let $\bm{\rho_m}$ be a
 minimizing sequence. Define
 \begin{align*}
  \tilde \rho_i^m(x) = R^2\rho_i^m(Rx), \ x \in \rt, R>0 \ .
 \end{align*}
Direct computation gives
\begin{align*}
 \mathcal{F}_0( \bm{\tilde \rho}_m) = \mathcal{F}_0( \bm{\rho}^m) + \left(\frac{1}{R^2} - 1\right)
 \sum_{i=1}^n \frac{1}{2}\int_{\rt}|x|^2\rho_i^m.
\end{align*}
Thus we have (using $\liminf(a_m + b_m) = \lim a_m + \liminf b_m,$ if $a_m$ converges)
\begin{align*}
 \inf_{\bm{\rho} \in \Gamma^{\bm{\beta}}} \mathcal{F}_0(\bm{\rho})
 \leq \lim \mathcal{F}_0( \bm{\rho}^m) + \liminf \left(\frac{1}{R^2} - 1\right)
 \sum_{i=1}^n \frac{1}{2}\int_{\rt}|x|^2\rho_i^m.
\end{align*}
which gives
\begin{align} \label{con}
 \liminf \left(\frac{1}{R^2} - 1\right)
 \sum_{i=1}^n \frac{1}{2}\int_{\rt}|x|^2\rho_i^m \geq 0.
\end{align}
Choosing $R<1$ in \eqref{con} we get $\liminf \sum_{i=1}^n \left(\frac{1}{2}\int_{\rt}|x|^2\rho_i^m \right)\geq 0,$ while
$R>1$ gives $\limsup \left(\sum_{i=1}^n \frac{1}{2}\int_{\rt}|x|^2\rho_i^m\right) \leq 0$ and hence
\begin{align*}
 \lim \left(\sum_{i=1}^n \frac{1}{2}\int_{\rt}|x|^2\rho_i^m\right) = 0.
\end{align*}
Therefore all the components of $\bm{\rho}^m$ concentrates at the origin and hence there does not exists a minimizer of
$\mathcal{F}_0$ on $\Gamma^{\bm{\beta}}.$
\end{proof}

\subsection{A Functional inequality:}
\begin{lem}
 The following inequality holds true
 \begin{align} \label{funct ineq}
  \inf_{\bm{\rho} \in \Gamma^{\bm{\beta}}} \mathcal{F}_{\bm{v}}(\bm{\rho})
  \leq \inf_{\bm{\rho} \in \Gamma^{\bm{\beta}}} \mathcal{F}_0(\bm{\rho}) + \min_{x_0 \in \rt} \sum_{i=1}^n
\frac{1}{2}\beta_i |x_0 - v_i|^2.
  \end{align}
\end{lem}
\begin{proof} 
 Let $\bm{\rho_m}$ be a
 minimizing sequence for $\inf_{\bm{\rho} \in \Gamma^{\bm{\beta}}} \mathcal{F}_0(\bm{\rho}).$ Define
 for $x_0 \in \rt,$
 \begin{align*}
  \bm{\tilde \rho}_m(x) = \bm{\rho}^m(x-x_0), \ x \in \rt.
 \end{align*}
Then a direct computation gives
\begin{align} \label{inq}
\inf_{\bm{\rho} \in \Gamma^{\bm{\beta}}} \mathcal{F}_{\bm{v}}(\bm{\rho}) \leq
 \mathcal{F}_{\bm{v}}(\bm{\tilde \rho}_m) &=  \mathcal{F}_0(\bm{\rho}^m) +
 \sum_{i=1}^n\frac{1}{2} \int_{\rt} \left(|x + x_0 - v_i| - |x|^2\right)\rho_i^m \notag \\
 &= \mathcal{F}_0(\bm{\rho}^m)
 + \sum_{i=1}^n\frac{1}{2} \int_{\rt} \langle x,x_0 -v_i\rangle\rho_i^m
 + \sum_{i=1}^n\frac{1}{2}\beta_i|x_0 - v_i|^2.
\end{align}
Since by Lemma \ref{concentration} $\lim \left(\sum_{i=1}^n \frac{1}{2}\int_{\rt}|x|^2\rho_i^m\right)= 0$
we get
\begin{align*}
 \left|\sum_{i=1}^n\frac{1}{2} \int_{\rt} \langle x,x_0 -v_i\rangle\rho_i^m\right|
 \leq \sum_{i=1}^n\frac{1}{2}\beta_i^{\frac{1}{2}}|x_0 - v_i|\left(\int_{\rt}|x|^2\rho_i^m\right)^{\frac{1}{2}}
 \rightarrow 0,
\end{align*}
as $m\rightarrow \infty.$ Therefore letting $m \rightarrow \infty$ in \eqref{inq} we get
\begin{align*}
 \inf_{\bm{\rho} \in \Gamma^{\bm{\beta}}} \mathcal{F}_{\bm{v}}(\bm{\rho})
 \leq \inf_{\bm{\rho} \in \Gamma^{\bm{\beta}}} \mathcal{F}_0(\bm{\rho}) + \sum_{i=1}^n\frac{1}{2}\beta_i|x_0 - v_i|^2.
\end{align*}
Since $x_0 \in \rt$ is arbitrary the proof of the lemma is completed.
\end{proof}

\begin{rem}
 If the equality occurs in \eqref{funct ineq} then every minimizing sequence for $\mathcal{F}_{\bm{v}}$ on $\Gamma^{\bm{\beta}}$
 is also a minimizing sequence for $\mathcal{F}_0$ on $\Gamma^{\bm{\beta}}.$
 Hence, for any such minimizing sequence we get
 $\sum_{i=1}^n \int_{\rt} \rho_i^m\ln \rho_i^m \rightarrow \infty$. Otherwise, as in Theorem \ref{sub existence} (Remark 3) we can prove the existence of a minimizer of $\mathcal{F}_0$
 on $\Gamma^{\bm{\beta}},$ which contradicts Lemma \ref{concentration}.
\end{rem}


\section{Blow up analysis: Brezis Merle type argument}
We pose  the following alternatives
\begin{prop} \label{alternatives}
 Suppose $\bm{\beta}$ satisfies \eqref{critical}, then one of the following alternative holds: either
 \begin{itemize}
  \item[(a)] there exists a minimizer of $\mathcal{F}_{\bm{v}}$ over $\Gamma^{\bm{\beta}},$ or
  \item[(b)] equality holds in the functional inequality \eqref{funct ineq}.
  \end{itemize}
\end{prop}
For the  proof of this Proposition will need  the two  Lemmas below: 
\par
Let $\bm{\beta}_m$ be a sequence such that  $\bm{\beta}_m\nearrow\bm{\beta}$ and satisfies
\begin{align*}
 \Lambda_J(\bm{\beta}_m) >0, \ \mbox{for all} \ \phi \neq J \subset I.
\end{align*}
One can indeed choose such sequence $\bm{\beta}_m,$ see for example \cite[Lemma $(5.1)$ and equation $(5.4)$]{CSW}.
By Theorem \ref{sub existence} the infimum $\inf_{\bm{\rho}\in \Gamma^{\bm{\beta}_m}} \mathcal{F}_{\bm{v}}(\bm{\rho})$
is attained. Let us denote the minimizer by
$\bm{\rho}^m  \in \Gamma^{\bm{\beta}_m}.$

\begin{lem} \label{moment bound}
 The following holds
 \begin{align*}
  \sup_m \int_{\rt}|x|^2 \rho_i^m < +\infty, \ \mbox{for all}  \  i \ .
 \end{align*}

\end{lem}
\begin{proof}
 For each fixed $m$ define
 \begin{align*}
  \tilde \rho_i^m(x) = R_m^2\rho_i^m(R_mx), \
 \end{align*}
where $R_m>0$ to be decided later. A direct computations gives
\begin{align*}
 \mathcal{F}_{\bm{v}}(\bm{\tilde \rho}_m) =  \mathcal{F}_{\bm{v}}(\bm{\rho}^m) + f_m(R_m),
\end{align*}
where $f_m:(0,\infty) \rightarrow \mathbb{R}$ is defined by
\begin{align*}
 f_m(t)= a_m \ln t + \frac{b_m}{2t^2} + \frac{2c_m}{t} + d_m,
\end{align*}
and $a_m,b_m,c_m,d_m$ are defined as follows:
\begin{align*}
 a_m = \frac{1}{4\pi}\Lambda_I(\bm{\beta}_m) \rightarrow 0, \ \ 2c_m = -\sum_{i=1}^n \int_{\rt}\langle x , v_i\rangle \rho_i^m \notag \\
 b_m = \sum_{i=1}^n \int_{\rt}|x|^2\rho_i^m, \ \ d_m = -\frac{1}{2}\sum_{i=1}^n \int_{\rt}|x-v_i|^2\rho_i^m+
 \frac{1}{2} |v_i|^2\beta_i^m.
\end{align*}
One can easily verify that the following inequalities hold:
\begin{align}\label{cdestm}
 |c_m| &\leq \left(\sup_i \frac{\sqrt{n}}{2}|v_i|\beta_i^{\frac{1}{2}}\right) b_m^{\frac{1}{2}}, \\
 -\frac{1}{4}\sum_{i=1}^n|v_i|^2\beta_i^m - &d_m \leq b_m \leq -4d_m + 4 \sum_{i=1}^n|v_i|^2\beta_i^m.
\end{align}
Since $\bm{\rho}^m$ minimizes $\mathcal{F}_{\bm{v}}$ over
$\Gamma^{\bm{\beta}_m}$ we have
\begin{align*}
 \inf_{\bm{\rho} \in \Gamma^{\bm{\beta}_m}} \mathcal{F}_{\bm{v}}(\bm{\rho}) \leq
 \mathcal{F}_{\bm{v}}(\bm{\tilde \rho}_m) &=  \mathcal{F}_{\bm{v}}(\bm{\rho}^m) + f_m(R_m)
 =\inf_{\bm{\rho} \in \Gamma^{\bm{\beta}_m}} \mathcal{F}_{\bm{v}}(\bm{\rho}) + f_m(R_m),
\end{align*}
and therefore $f_m(R_m) \geq 0.$ We now choose $R_m$ such that $f_m(R_m)$ is the  minimum of $f_m(t)$ over $(0,\infty).$
Since $f_m(t) \rightarrow \infty$ as $t \rightarrow 0$ and $t \rightarrow\infty$ we have $R_m>0$ and satisfies
$f_m^{\prime}(R_m) = 0,$ i.e.,
\begin{align}\label{Rm}
 \frac{a_m}{R_m} - \frac{b_m}{R_m^3}-\frac{2c_m}{R_m^2}=0 \ . 
\end{align}
Therefore we have from \eqref{Rm}
\begin{align} \label{cm}
 f_m(R_m) &= a_m\ln R_m + a_m - \frac{b_m}{2R_m^2} + d_m \notag\\
 &=a_m\ln a_mR_m - a_m\ln a_m + a_m - \frac{b_m}{2R_m^2} + d_m.
\end{align}
Now assume, in contradiction, that $b_m \rightarrow \infty.$ Then we deduce from \eqref{Rm} and the bound on $c_m$ (\ref{cdestm}) that
$a_mR_m \leq Cb_m^{\frac{1}{2}},$ for some positive constant $C$ independent of $m$ and hence we have
\begin{align} \label{above}
 a_m\ln a_m R_m \leq a_m \ln C + \frac{a_m}{2}\ln b_m,
\end{align}
 Since $b_m \geq 0,$ \eqref{cm} and \eqref{above} and the estimates of $d_m$ in terms of $b_m$ gives
\begin{align*}
 f_m(R_m) &\leq (a_m \ln C - a_m\ln a_m + a_m) + \frac{a_m}{2}\ln b_m + d_m \\
 &\leq o(1) + \frac{a_m}{2}\ln b_m - \frac{b_m}{4} + \sum_{i=1}^n |v_i|^2\beta_i^m \\
 &= O(1) + \frac{a_m}{2}\ln b_m - \frac{b_m}{4} <0,
\end{align*}
for large $m.$ This contradicts the fact that $f_m(R_m) \geq 0,$ and hence the proof of the lemma is completed.
\end{proof}

\begin{lem} \label{min conv}
The followings hold true:
\begin{itemize}
 \item[(a)]
 \begin{align} \label{inf ineq1}
 \lim_{m\rightarrow \infty} \inf_{\bm{\rho} \in \Gamma^{\bm{\beta}_m}} \mathcal{F}_{\bm{v}}(\bm{\rho})
 \leq \inf_{\bm{\rho} \in \Gamma^{\bm{\beta}}} \mathcal{F}_{\bm{v}}(\bm{\rho})
 \end{align}
 \item[(b)]
\begin{align} \label{inf ineq2}
 \lim_{m\rightarrow \infty} \inf_{\bm{\rho} \in \Gamma^{\bm{\beta}_m}} \mathcal{F}_0(\bm{\rho})
 = \inf_{\bm{\rho} \in \Gamma^{\bm{\beta}}} \mathcal{F}_0(\bm{\rho})
\end{align}
\end{itemize}
\end{lem}
\begin{proof}
We first prove inequality \eqref{inf ineq1}. Let $\bm{\rho}  \in \Gamma^{\bm{\beta}}$ be a
fixed element. Choose $R_i^m >0$ such that $\int_{B(0,R_i^m)} \rho_i = \beta_i^m$ and define
 $\rho_i^m = \rho_i \chi_{B(0,R_i^m)}.$ Then $\bm{\rho}^m \in \Gamma^{\bm{\beta}_m}$ and
 by dominated convergence theorem
 \begin{align*}
  \lim_{m \rightarrow \infty} \mathcal{F}_{\bm{v}}(\bm{\rho}^m) = \mathcal{F}_{\bm{v}}(\bm{\rho}).
 \end{align*}
Thus we have
\begin{align*}
 \lim_{m \rightarrow \infty} \inf_{\bm{\rho}\in \Gamma^{\bm{\beta}_m}}\mathcal{F}_{\bm{v}}(\bm{\rho})
 \leq \lim_{m \rightarrow \infty} \mathcal{F}_{\bm{v}}(\bm{\rho}^m) = \mathcal{F}_{\bm{v}}(\bm{\rho}).
\end{align*}
Since $\bm{\rho} \in \Gamma^{\bm{\beta}}$ is arbitrary, we have proved the inequality \eqref{inf ineq1}.
Next we prove \eqref{inf ineq2}. Thanks to \eqref{inf ineq1}, we only need to show
$\lim_{m\rightarrow \infty} \inf_{\bm{\rho} \in \Gamma^{\bm{\beta}_m}} \mathcal{F}_0(\bm{\rho})
 \geq \inf_{\bm{\rho} \in \Gamma^{\bm{\beta}}} \mathcal{F}_0(\bm{\rho}).$ This step is a little bit technical and
therefore we divide the proof into several parts.

 (1) By Theorem \ref{sub existence}, there exists $\bm{\rho}^m  \in \Gamma^{\bm{\beta}_m}$ such that
 \begin{align*}
  \mathcal{F}_0(\bm{\rho}^m) = \inf_{\bm{\rho} \in \Gamma^{\bm{\beta}_m}} \mathcal{F}_0(\bm{\rho}).
 \end{align*}
Furthermore, $\rho_i^m$ are radially symmetric and decreasing function of $r=|x|.$ By abuse of notations, we will
also denote the radial function by $\rho_i^m(r).$

(2) A simple adoption of the proof of Lemma \ref{concentration} gives $\int_{\rt} |x|^2\rho_i^m(x) \rightarrow 0$
as $m \rightarrow \infty.$ Therefore for any $r \in (0,\infty)$
\begin{align*}
 o_m(1) = \int_{\rt} |x|^2\rho_i^m(x) = 2\pi \int_0^{\infty} s^3\rho_i^m(s) \ ds
 \geq 2\pi \int_0^r s^3\rho_i^m(s) \ ds \geq \frac{\pi}{2}r^4 \rho_i^m(r),
\end{align*}
where $o_m(1)$ denotes a quantity going to $0$ as $m \rightarrow \infty.$
Thus we have $\sup_{r \in (0,\infty)} r^4\rho_i^m(r) = o_m(1)$ as $m \rightarrow \infty.$

 A similar argument using $\int_{\rt} \rho_i^m = \beta_i^m$ gives $\sup_{r \in (0,\infty)} r^2\rho_i^m(r)
\leq \frac{\beta_i^m}{\pi}.$

(3) Let $\phi$ be a smooth, nonnegative, radial, compactly supported function such that $\int_{\rt} \phi = 1.$ Define
$\epsilon_i^{(m)} = \beta_i - \beta_i^m >0$ and
\begin{align*}
 \tilde \rho_i^m(x) = \rho_i^m(x) + \epsilon_i^{(m)} \phi(x), \ x \in \rt .
\end{align*}
Then $\bm{\tilde \rho}_m  \in \Gamma^{\bm{\beta}}$ for all $m$ and hence
$\inf_{\bm{\rho} \in \Gamma^{\bm{\beta}}} \mathcal{F}_0(\bm{\rho}) \leq
\mathcal{F}_0(\bm{\tilde \rho}_m).$ Now we will estimate each term of $\mathcal{F}_0(\bm{\tilde \rho}_m)$
and show that
\begin{align*}
 \mathcal{F}_0(\bm{\tilde \rho}_m) = \mathcal{F}_0(\bm{\rho}^m) + o_m(1).
 \end{align*}
(4)
\begin{align} \label{q1}
\int_{\rt} \tilde \rho_i^m \ln \tilde \rho_i^m - \int_{\rt} \rho_i^m \ln \rho_i^m = o_m(1).
\end{align}
\
Let us denote by $k_m := \max_{1\leq i \leq n}\max\{\sup_{r \in (0,\infty)} r^4\rho_i^m(r), \sup_{r \in (0,\infty)} r^4\tilde\rho_i^m(r)\},$
then using $(2)$ and $\epsilon_i^{(m)} \rightarrow 0,$ we see that $k_m \rightarrow 0.$
Let $\delta_m$ be a sequence such that $\delta_m \rightarrow 0$ and
$k_m\ln k_m/\delta_m^3 \rightarrow 0.$ Clearly we have
\begin{align*}
 \int_{B(0,\delta_m)} (\tilde \rho_i^m \ln \tilde \rho_i^m -  \rho_i^m \ln \rho_i^m) \chi_{\{\rho_i^m \leq 2\}} = o_m(1),
\end{align*}
because $t\ln t$ is bounded on any compact subset of $[0,\infty).$ Now using mean value theorem we get
\begin{align*}
 &\int_{B(0,\delta_m)} (\tilde \rho_i^m \ln \tilde \rho_i^m -  \rho_i^m \ln \rho_i^m) \chi_{\{\rho_i^m > 2\}} \\
 =& 2\pi \epsilon_i^{(m)}  \int_0^{\delta_m}\int_0^1 r \left[1 + \ln (\rho_i^m(r) + t\epsilon_i^{(m)}\phi(r))\right]
 \phi(r) \chi_{\{\rho_i^m > 2\}} \ dtdr \\
 =& o_m(1) + 2\pi \epsilon_i^{(m)}\int_0^1\int_0^{\delta_m} r \ln (\rho_i^m(r) + t\epsilon_i^{(m)}\phi(r))
 \phi(r) \chi_{\{\rho_i^m > 2\}} \ drdt
\end{align*}
On the set $\{\rho_i^m > 2\},$ we have $\rho_i^m(r) + t\epsilon_i^{(m)}\phi(r) >1$ for $m$ large enough. Moreover, using the estimate
of (2) we see that $\rho_i^m(r) + t\epsilon_i^{(m)}\phi(r) \leq \frac{C}{r^2}$ where $C$ is some positive constant. Therefore
$ 0 \leq r\ln (\rho_i^m(r) + t\epsilon_i^{(m)}\phi(r)) \leq r \ln \frac{C}{r^2}$ and hence
\begin{align*}
 2\pi \epsilon_i^{(m)}\int_0^1\int_0^{\delta_m} r \ln (\rho_i^m(r) + t\epsilon_i^{(m)}\phi(r))
 \phi(r) \chi_{\{\rho_i^m > 2\}} \ drdt = o_m(1),
\end{align*}
which gives
\begin{align*}
 \int_{B(0,\delta_m)} (\tilde \rho_i^m \ln \tilde \rho_i^m -  \rho_i^m \ln \rho_i^m) \chi_{\{\rho_i^m > 2\}}= o_m(1).
\end{align*}

 Now let us estimate $\int_{B(0,\delta_m)^c} \rho_i^m \ln \rho_i^m.$
\begin{align*}
 \left|\int_{B(0,\delta_m)^c} \rho_i^m \ln \rho_i^m \right|&=
 \left|2\pi \int_{\delta_m}^{\infty} r \rho_i^m(r) \ln \rho_i^m (r) \ dr \right|\\
 &\leq 2\pi \int_{\delta_m}^{\infty} \frac{|r^4\rho_i^m \ln (r^4\rho_i^m)|}{r^3} \ dr
 + 8\pi \int_{\delta_m}^{\infty} \frac{r^4\rho_i^m |\ln r| }{r^3} \ dr \\
 &\leq 2\pi |k_m\ln k_m| \int_{\delta_m}^{\infty} \frac{dr}{r^3}
 + 8 \pi k_m \int_{\delta_m}^{\infty} \frac{|\ln r|}{r^3} dr \\
 & \leq \frac{2\pi |k_m\ln k_m|}{\delta_m^2} + C \frac{k_m}{\delta_m^{2- \epsilon}}, \ \mbox{for some $\epsilon >0$}\\
 &=o_m(1).
 \end{align*}
 In an entirely similar way we can verify that $\left|\int_{B(0,\delta_m)^c} \tilde\rho_i^m \ln \tilde\rho_i^m \right| = o_m(1),$
 and hence we have proved \eqref{q1}.

(5) Next we estimate
\begin{align}\label{q2}
 &\int_{\rt} \int_{\rt}\tilde \rho_i^m (x) \ln |x-y| \tilde \rho_j^m(y) \notag\\
 =&
  \int_{\rt} \int_{\rt} \rho_i^m (x) \ln |x-y|  \rho_j^m(y)
  + \epsilon_i^m  \int_{\rt} \int_{\rt} \phi (x) \ln |x-y| \rho_j^m(y) \notag \\
  &+ \epsilon_j^m  \int_{\rt} \int_{\rt} \rho_i^m (x) \ln |x-y|  \phi(y)
  + \epsilon_i^m \epsilon_j^m  \int_{\rt} \int_{\rt} \phi(x) \ln |x-y| \tilde \phi(y) \notag \\
  =&\int_{\rt} \int_{\rt} \rho_i^m (x) \ln |x-y|  \rho_j^m(y) + o_m(1).
\end{align}
Where we have used the fact that $|\int_{\rt} \ln|x - y|\phi(y) \ d^2y | \leq C(1 + \ln (1+|x|))$ for
all $x \in \rt.$

(7) Finally we have
\begin{align} \label{q3}
 \int_{\rt} |x|^2 \tilde \rho_i^m(x) = \int_{\rt} |x|^2  \rho_i^m(x) + o_m(1).
\end{align}
Combining \eqref{q1}, \eqref{q2} and \eqref{q3} we get
\begin{align*}
 \inf_{\bm{\rho} \in \Gamma^{\beta}}  \mathcal{F}_0(\bm{\rho})\leq
 \mathcal{F}_0(\bm{\tilde \rho}_m) = \mathcal{F}_0(\bm{\rho}^m) + o_m(1) =
 \inf_{\bm{\rho} \in \Gamma^{\bm{\beta}_m}}  \mathcal{F}_0(\bm{\rho}) + o_m(1).
\end{align*}
Letting $m \rightarrow \infty,$ we reach at the desired conclusion. This completes the proof of the lemma.
\end{proof}

\subsection{Proof of Proposition \ref{alternatives}}

Recall that $\bm{\rho}^m$ is a minimizer of $\mathcal{F}_{\bm{v}}$ over $\Gamma^{\bm{\beta}_m}$, where $\bm{\beta}_m\nearrow\bm{\beta}$.  Define the Newtonian
potentials
\begin{align*}
 u_i^m(x) =
 -\frac{1}{2\pi}\int_{\rt} \ln |x-y|\rho_i^m(y) \ d^2y, \ x \in \rt.
\end{align*}
By variational principle and Lemma \ref{moment bound}, $u_i^m$ satisfies the following equation:
\begin{align*}
\begin{cases}
-\Delta u_i^m(x) = \mu_i^m e^{\sum_{j=1}^na_{ij}u_j^m(x) - \frac{1}{2}|x-v_i|^2}, \ \mbox{in} \ \rt,\\
\mu_i^m \int_{\rt}e^{\sum_{j=1}^na_{ij}u_j^m - \frac{1}{2}|x-v_i|^2} = \beta_i^m,\\
\mu_i^m \int_{\rt}|x|^2e^{\sum_{j=1}^na_{ij}u_j^m - \frac{1}{2}|x-v_i|^2} \leq C_0,
\end{cases}
\end{align*}
where $C_0$ is a constant independent of $m.$ Define
\begin{align*}
v_i^m(x) = \ln \mu_i^m + \sum_{j=1}^na_{ij}u_j^m(x), \ x \in \rt.
\end{align*}
Let us consider the two cases:
\par \noindent
\textbf{Case (A):} Suppose there exists $R>0$ such that
\begin{align} \label{sup}
\max_{1 \leq i \leq n} \sup_{x\in B(0,R)} v_i^m(x) \rightarrow \infty, \ \mbox{as} \ m \rightarrow \infty.
\end{align}
\noindent
\textbf{Case (B):} For any $R>0$ there exists a constant $C(R)$ such that
\begin{align*}
\max_{1 \leq i \leq n}\sup_{x \in B(0,R)} v_i^m(x) \leq C(R).
\end{align*}

We first prove:
\begin{lem}
	Under the assumption of Case (A), the following equality holds:
	\begin{align} \label{funct ineq 3}
	\inf_{\bm{\rho} \in \Gamma^{\bm{\beta}}} \mathcal{F}_{\bm{v}}(\bm{\rho})
	= \inf_{\bm{\rho} \in \Gamma^{\bm{\beta}}} \mathcal{F}_0(\bm{\rho}) + \min_{x_0 \in \rt} \sum_{i=1}^n
	\frac{1}{2}\beta_i |x_0 - v_i|^2.
	\end{align}
\end{lem}

\begin{proof}

By definition $v_i^m, 1 \leq i \leq n$ satisfies the equation
\begin{align*}
 \begin{cases}
  -\Delta v_i^m(x) = \sum_{j=1}^na_{ij} e^{v_j^m(x) - \frac{1}{2}|x-v_i|^2}, \ \mbox{in} \ \rt,\\
   \int_{\rt}e^{v_i^m - \frac{1}{2}|x-v_i|^2} = \beta_i^m,\\
  \int_{\rt}|x|^2e^{v_i^m - \frac{1}{2}|x-v_i|^2} \leq C_0.
 \end{cases}
\end{align*}

Furthermore, the following relation holds:
\begin{align} \label{rhouv}
 \rho_i^m(x) =
 \mu_i^m e^{\sum_{j=1}^na_{ij}u_j^m(x) - \frac{1}{2}|x-v_i|^2}
 = e^{v_i^m(x) - \frac{1}{2}|x-v_i|^2}, \ x \in \rt .
 \end{align}
After passing to a subsequence if necessary  we may assume the supremum in \eqref{sup}
is attained by $v_1^m$ for all $m.$ That is, there exists $x_m \in \overline{B(0,R)}$ such that
\begin{align*}
v_1^m(x_m) = \max_{i} \sup_{x\in B(0,R)} v_i^m(x) \rightarrow \infty, \ \mbox{as} \ m \rightarrow \infty.
\end{align*}
Let $x_m \rightarrow x_0$ for some $x_0 \in \overline{B(0,R)},$ and choose a $\tilde R >0$ large enough so that
$\overline{B(0,R)} \subset B(x_0,\tilde R).$ Since $v_1^m(x_m) \rightarrow \infty$ we have
\begin{align} \label{sup1}
 \sup \{v_i^m(x) + 2 \ln (\tilde R - |x-x_0|) \ : \ x \in B(x_0,\tilde R), 1 \leq i \leq n\} \rightarrow \infty, \
 \mbox{as} \ m \rightarrow \infty.
\end{align}
Again after passing to a subsequence me may assume $y_m \in B(x_0,\tilde R)$ be the point and $i_0$ be the index such
that the supremum in \eqref{sup1} is attained for all $m.$ Since $2 \ln (\tilde R - |x-x_0|)$ is
bounded above on $ B(x_0,\tilde R)$ we have $v_{i_0}^m(y_m) \rightarrow \infty.$

Define $\delta_m = e^{-\frac{v_{i_0}^m(y_m)}{2}},$ then $\delta_m \rightarrow 0$ and it follows from \eqref{sup1} that
\begin{align} \label{def}
 \left(\frac{\tilde R - |y_m - x_0|}{\delta_m}\right) \rightarrow \infty, \ \mbox{as} \ m \rightarrow \infty.
\end{align}
Now define
\begin{align*}
 \tilde v_i^m(x) = v_i^m(y_m + \delta_m(x-x_0)) + 2 \ln \delta_m.
\end{align*}
We note that $\tilde v_{i_0}^m(x_0) = 0$ for all $m.$  Furthermore, it follows from \eqref{def} that for any $M >0$
fixed and $x \in B(x_0,M), y_m + \delta_m(x-x_0) \in B(x_0, \tilde R)$ for large $m.$ Now $\tilde v_i^m(x)$ satisfies
the equation
\begin{align}\label{222}
 \begin{cases}
  -\Delta \tilde v_i^m(x) = \sum_{j=1}^n a_{ij} e^{\tilde v_j^m(x) - \frac{1}{2}|y_m + \delta_m(x-x_0) -v_j|^2} \ \mbox{in}
  \ \rt, \\
  \int_{\rt} e^{\tilde v_i^m(x) - \frac{1}{2}|y_m + \delta_m(x-x_0) -v_i|^2} = \beta_i^m.
 \end{cases}
\end{align}

Let $y_m \rightarrow y_0 \in \overline{B(x_0, \tilde R)}.$
Since $\tilde v_{i_0}^m(x_0) = 0$ either  $\tilde v_i^m$ converges to some $\tilde v_i$ in $C^0_{loc}(\rt)$ for all $i$  or
$\tilde v_i^m$ converges to $-\infty$
uniformly on compact subsets of $\rt$ for some $i\not=i_0$.

Let $I^{'}\subset I$ is the set of indices such that  $\tilde v_i\not=-\infty$ iff $i\in I^{'}$. Then
$\tilde v_{i}^m$ converges to $\tilde v_{i}$ in $C^0_{loc}(\rt)$ for $i\in I^{'}$
and, by \eqref{222}
\begin{align} \label{n}
 \begin{cases}
  -\Delta \tilde v_i = \sum_{j\in I^{'}}a_{ij} e^{\tilde v_j - \frac{1}{2}|y_0-v_i|^2} \  \ \mbox{in} \ \rt,\\
   \int_{\rt}e^{\tilde v_i - \frac{1}{2}|y_0-v_i|^2} = \tilde \beta_i, \
  \end{cases}
\end{align}
Letting $z_i(x)=\tilde v_i(x)-\frac{1}{2}|y_0-v_i|^2$ we obtain
\begin{align} \label{z}
 \begin{cases}
  -\Delta z_i = \sum_{j\in I^{'}}a_{ij} e^{z_j } \  \ \mbox{in} \ \rt,\\
   \int_{\rt}e^{z_i} = \tilde \beta_i  \ .
  \end{cases}
\end{align}

holds for $i\in I^{'}$ for some $\tilde \beta_i \leq \beta_i.$

 A necessary condition for the existence of solution to \eqref{z} is
$\Lambda_{I^{'}}(\bm{\tilde\beta}) =0$ (\cite{CSW}, see also \cite{LZ11,PT14}). Since we assumed $\Lambda_I(\bm{\beta})=0$ this implies $I^{'}=I$ and $\tilde{\bm{\beta}}=\bm\beta$.
 (see \cite{CSW}).


 It follows that, in Case (A), $\rho_i^m$ concentrates at some point $y_0 \in \rt.$ In particular
 \begin{align} \label{fun ineq}
 \lim_{m\rightarrow\infty} \int_{\rt}|x-v_i|^2\rho_i^m(x)d^2x\geq \beta_i|y_0-v_i|^2 
 \  \ \text{for all} \ 1\leq i\leq n \ . 
 \end{align}
 We want to show that $y_0$
 is the global minima of $\sum_{i=1}^n
\frac{1}{2}\beta_i |x - v_i|^2$ on $\rt.$ Let us define $\bm{\tilde{\rho}}_m$ as
\begin{align*}
 \bm{\tilde{\rho}}_m(x) = \frac{1}{\delta^2}\bm{\rho}^m\left(\frac{x}{\delta} - y_0\right).
\end{align*}
Then
\begin{align*}
 \mathcal{F}_0(\bm{\tilde \rho}_m) &= \mathcal{F}_{\bm{v}}(\bm{\rho}^m) - \frac{\Lambda_{I}(\bm{\beta}_m)}{4\pi}\ln \delta
 + \delta^2 \sum_{i=1}^n \frac{1}{2}\int_{\rt} |x+y_0|^2\rho_i^m \\
 &- \sum_{i=1}^n \frac{1}{2}\int_{\rt} |x-v_i|^2\rho_i^m.
\end{align*}
Therefore we obtain
\begin{align*}
 \inf_{\bm{\rho} \in \Gamma^{\bm{\beta}_m}} \mathcal{F}_0(\bm{\rho})
 \leq \mathcal{F}_{\bm{v}}(\bm{\rho}^m) - \frac{\Lambda_{I}(\bm{\beta}_m)}{4\pi}\ln \delta
 + \delta^2 O(1)
 &- \sum_{i=1}^n \frac{1}{2}\int_{\rt} |x-v_i|^2\rho_i^m.
 \end{align*}

Letting $m\rightarrow \infty$ and using \eqref{fun ineq} and Lemma \ref{min conv}(b) we get
\begin{align*}
 \inf_{\bm{\rho} \in \Gamma^{\bm{\beta}}} \mathcal{F}_0(\bm{\rho}) \leq
 \inf_{\bm{\rho} \in \Gamma^{\bm{\beta}}} \mathcal{F}_{\bm{v}}(\bm{\rho})
 + \delta^2 O(1)
 -\sum_{i=1}^n \frac{1}{2}\beta_i |y_0-v_i|^2.
\end{align*}
Since $\delta>0$ is arbitrary, by \eqref{funct ineq} we get $y_0$
 is the global minima of $\sum_{i=1}^n
\frac{1}{2}\beta_i |x - v_i|^2$ on $\rt$ and \eqref{funct ineq 3} holds true.
\end{proof}

\begin{lem}
Under the assumption of Case (B) there exists a minimizer of $\mathcal{F}_{\bm{v}}$ in	$ \Gamma^{\bm{\beta}}$. In particular
	\begin{align} \label{funct ineq 2}
	\inf_{\bm{\rho} \in \Gamma^{\bm{\beta}}} \mathcal{F}_{\bm{v}}(\bm{\rho})
	< \inf_{\bm{\rho} \in \Gamma^{\bm{\beta}}} \mathcal{F}_0(\bm{\rho}) + \min_{x_0 \in \rt} \sum_{i=1}^n
	\frac{1}{2}\beta_i |x_0 - v_i|^2.
	\end{align}
	
\end{lem}
\begin{proof}
Under this assumption, we have from \eqref{rhouv} that $||\rho_i^m||_{L^{\infty}(B(0,R))} \leq C_0,$ for some constant $C_0$
independent of $m.$ In the proof $C_0$ will stand for some universal constant independent of $m$
but may depend on $R.$ Then
\begin{align} \label{bbd rho1}
 \left|\sum_{i=1}^n \int_{B(0,R)} \rho_i^m(x) \ln \rho_i^m(x) \ d^2x\right| \leq C_0.
\end{align}

Now let
\begin{align*}
 \tilde u_i^m(x) := -\frac{1}{2\pi} \int_{\rt} \ln|x-y| \rho_i^m(y)\chi_{B(0,R)}(y) \ d^2y,
\end{align*}
then it follows from Lemma \ref{asymptotic behavior} (using the fact $||\rho_i^m||_{L^{\infty}(B(0,R))} \leq C_0$) that
\begin{align*}
 |\tilde u_i^m(x)| \leq
 \begin{cases}
  C_0, \  \ \ \ \ \ \ \ \ \ \ \ \ \ \ \ \mbox{if} \ |x| \leq 1,\\
  C_0(1+\ln|x|), \ \mbox{if} \ |x| > 1.
 \end{cases}
\end{align*}
Thus we have
\begin{align} \label{bbd rho2}
 \left|\int_{B(0,R)^c} \int_{B(0,R)} \rho_i^m(x) \ln |x-y| \rho_j^m(y) \ d^2y d^2x\right|
 &\leq \int_{\rt} \rho_i^m (x)|\tilde u_i^m(x) |\ d^2x \notag \\
 &\leq C_0 \left[\int_{\rt} \rho_i^m \ d^2x + \int_{\{|x|>1\}} \ln|x|\rho_i^m  \ d^2x\right] \notag \\
 &\leq C_0 \left[\beta_i^m + \int_{\rt} |x|^2\rho_i^m  \ d^2x\right]
 \leq C_0.
\end{align}
Let us define $\hat{\bm{\rho}}_m^R(x) = \bm{\rho}^m(x)\chi_{B(0,R)^c}(x)$.
Let
 \begin{multline}\label{xx}
 \mathcal{F}_{\bm{v},R}(\bm{\rho}):=\sum_{i=1}^n\int_{B(0,R)}\rho_i\ln\rho_i+ \frac{1}{4\pi} \sum_{i=1}^n\sum_{j=1}^na_{ij} \int_{B(0,R)} \int_{B(0,R)} \rho_i^m(x) \ln |x-y| \rho_j^m(y)
\\ + \frac{1}{2}\sum_{i=1}^n\int_{B(0,R)}|x-v_i|^2\rho_i \ . \end{multline}

We can write $\mathcal{F}_{\bm{v}}(\bm{\rho}^m)$
as
\begin{multline}\label{yy}
\mathcal{F}_{\bm{v}}(\bm{\rho}^m)=\mathcal{F}_{\bm{v},R}(\bm{\rho}^m)+ \mathcal{F}_{\bm{v}}(\hat{\bm{\rho}}^R_m) \\
+ \frac{1}{2\pi}\sum_{i=1}^n\sum_{j=1}^n a_{ij}\int_{B(0,R)}\int_{B(0,R)^c}\rho_i^m(x)\ln|x-y|\rho_j^m(y)d^2xd^2y \ .
\end{multline}
Since $\|\bm{\rho}^m\|_{L^{\infty}(B(0,R))}\leq C_0$ we obtain that $\mathcal{F}_{\bm{v},R}(\bm{\rho}^m)=O(1)$. Also, \eqref{bbd rho2} implies that the second line in \eqref{yy} is $O(1)$ as well. 
Since $\mathcal{F}_{\bm{v}}(\bm{\rho}^m)$ is a bounded sequence (as $\bm{\rho}^m$ is a minimizer of $\inf_{\Gamma^{\bm{\beta}_m}}
\mathcal{F}_{\bm{v}},$ see Lemma \ref{min conv}(a)) this implies that
\begin{equation}\label{uim}  \mathcal{F}_{\bm{v}}(\hat{\bm{\rho}}^R_m)=O(1) \end{equation}
 uniformly in $m$.
\par
Next, observe that we can choose $R$ large enough for which $\int_{\rt}\hat{\bm{\rho}}^R_m<\bm{\beta}/2$.
Indeed, since $\int_{\rt}|x|^2\rho_i^m\leq C$ then $\int_{\{|x|>R\}}\rho_i^m\leq R^{-2}\int_{\{|x|>R\}}|x|^2\rho_i^m\leq C/R^2$. For such $R$, $\hat{\bm{\rho}}_m$ is sub-critical, uniformly in $m$, thus
$$ \mathcal{F}_{\bm{v}}(\hat{\bm{\rho}}^R_m)\geq C\sum_{i=1}^n\int_{\rt} \hat\rho_i^m\ln\hat\rho_i^m \ . $$
From \eqref{uim} we obtain that $\hat{\bm{\rho}}_m^R$ has a uniform bound in $\mathbb{L}\ln\mathbb{L}$. 
Since $\|\bm{\rho}^m\|_{L^{\infty}(B(0,R))}=O(1)$ by assumption we obtain that $\bm{\rho}^m$ is bounded in $\mathbb{L}\ln\mathbb{L}$ as well.

Proceeding as in the sub critical case (Theorem \ref{sub existence}, see Remark 3) we can prove the existence of a minimizer
of $\mathcal{F}_{\bm{v}}$ over $\Gamma^{\bm{\beta}}.$
\end{proof}

\subsection{Case of \ \texorpdfstring{$Var(v_1, \ldots, v_n)$} \ \ large: Proof of Theorem \ref{main}-d}

According to Proposition \ref{alternatives} we only have to exclude case A. We show it in the case $n=2$. The general case follows similarly. 
\begin{lem}
 Suppose $\bm{\beta}$ satisfies \eqref{critical}.  Then there exists a constant $\kappa(\bm{\beta})$ such that whenever
  $ |v_1-v_2|>\kappa$, then \eqref{funct ineq 2} holds.
\end{lem}
\begin{proof}
Let $\bar\rho$ be any non-negative, bounded function of compact support (say $\bar\rho(x)=0$ if $|x|>1$) such that $\int_{\rt}\bar\rho=1$.
Define $\rho_i(x):= \beta_i\bar\rho(x-v_i)$
so that $\bm{\rho}\in \Gamma^{\bm{\beta}}$. 

 Then we immediately see that
 \begin{align*}
  &\left|\int_{\rt}  \rho_i \ln  \rho_i\right| = O(1), \ \
  \left|\int_{\rt} \int_{\rt}  \rho_i(x) \ln |x-y|  \rho_i(y)\right| =O(1), \\
  &\int_{\rt} |x-v_i|^2 \rho_i = \int_{\rt} |x|^2\bar \rho = O(1),
 \end{align*}
for all $i=1,2,$ where $O(1)$ denotes a quantity independent of $v_i.$ Now
 \begin{align} \label{e1}
  \int_{\rt} \int_{\rt}  \rho_1(x) \ln |x-y| \rho_2(y)
  = \beta_1\beta_2\int_{\rt} \int_{\rt} \bar\rho(x) \ln |x-y+(v_1-v_2)| \bar \rho(y).
 \end{align}
One can easily estimate that
$|\ln |x-y + (v_1-v_2)| - \ln |v_1-v_2|| \leq \frac{2}{|v_1-v_2|-2},$ for all $x,y \in (0,1)$ provided $|v_1-v_2|>2$
(this condition on $|v_1-v_2|$ is unnecessary, because we can choose the support of $\bar\rho$ accordingly).
Since $\bar\rho$ has support in $B(0,1)$ we get
\begin{align} \label{e2}
 &\int_{\rt} \int_{\rt} \bar\rho(x) \ln |x-y+(v_1-v_2)|  \bar\rho(y)- \ln |v_1-v_2| \notag \\
& = \int_{\rt} \int_{\rt} \bar\rho(x) \left(\ln |x-y+(v_1-v_2)| -  \ln |v_1-v_2| \right)\bar\rho(y)
 = O(1).
\end{align}
Thus we obtain from \eqref{e1} and \eqref{e2},
\begin{align} \label{e3}
 \inf_{\bm{\rho} \in \Gamma^{\bm{\beta}}} \mathcal{F}_{\bm{v}}(\bm{\rho})
 \leq \mathcal{F}_{\bm{v}}(\bm{\tilde \rho}) = O(1) + \frac{a_{12}}{2\pi}\beta_1\beta_2 \ln |v_1-v_2|.
\end{align}
While the right hand side of \eqref{funct ineq} becomes
\begin{align} \label{e4}
 \inf_{\bm{\rho} \in \Gamma^{\bm{\beta}}} \mathcal{F}_0(\bm{\rho}) + \min_{x_0 \in \rt} \sum_{i=1}^n \frac{\beta_i}{2}
 |x_0 - v_i|^2
 = O(1) + \frac{\beta_1\beta_2}{2(\beta_1+\beta_2)} |v_1 - v_2|^2.
\end{align}
We see from \eqref{e3} and \eqref{e4} that the equality can not occur in \eqref{funct ineq} provided $|v_1 - v_2|$
is very large. Hence by Proposition \ref{alternatives}, there exists a minimizer of $\mathcal{F}_{\bm{v}}$
on $\Gamma^{\bm{\beta}}.$ This completes the proof of the lemma.
\end{proof}

\noindent
\textbf{Proof of Theorem \ref{cor}:}

\begin{proof}
The proof of $(a)$ and $(c)$ follows from Theorem \ref{main} $(b)$ and $(d)$ respectively. We only need to prove $(b).$
Since $A$ is invertible and all the $v_i$ are equal by translating and adding constants to the solution we can assume 
 $u_i, 1 \leq i \leq n$ satisfies
 \begin{align} \label{ne}
  \begin{cases}
   -\Delta u_i = e^{\sum_{j=1}^n a_{ij} u_j - \frac{1}{2}|x|^2}, \ \mbox{in} \ \rt, \\
   \int_{\rt} e^{\sum_{j=1}^n a_{ij} u_j - \frac{1}{2}|x|^2} = \beta_i.
  \end{cases}
\end{align}
Again using the invertibility and irreducibility of $A$ we get by \cite[Proposition $4.1$]{CSW} with 
$V_i(x) = e^{-\frac{|x|^2}{2}}$ that $u_i$ in \eqref{ne} are radially symmetric with respect to the origin.
By abuse of notation we still denote the radial function by $u_i(r), r=|x|.$
Then $u_i$ satisfies
\begin{align} \label{ne1}
 -\frac{1}{r}\frac{d}{dr} \left(r \frac{du_i}{dr}\right) = e^{\sum_{j=1}^n a_{ij}u_j(r) - \frac{r^2}{2}}, r \in (0, \infty).
\end{align}
Define 
\begin{align*}
 m_i(r) = 2\pi \int_0^r se^{\sum_{j=1}^n a_{ij}u_j(s) - \frac{s^2}{2}} ds = -2\pi r\frac{du_i}{dr}, \ r \in (0,\infty), i = 1,\cdots,n.
\end{align*}
Then $m_i$ satisfies
\begin{align} \label{asy1}
 \lim_{r \rightarrow 0+} m_i(r) = 0, \lim_{r \rightarrow \infty} m_i(r) = \beta_i, \ \mbox{and} \
 m_i \ \mbox{are non decreasing}.
 \end{align}
Furthermore, since $u_i$ has log decay at infinity i.e., $|u_i(r) + \frac{\beta_i}{2\pi} \ln r| = O(1)$ as 
$r \rightarrow \infty$ (see \cite[Proposition 3.1]{CSW}) we see that
\begin{align} \label{asy2}
 \lim_{r \rightarrow \infty} r^2m_i^{\prime}(r) = 0.
\end{align}
Now define $w_i(s) = m_i(e^s), s \in (-\infty, \infty)$ then it follows from 
\eqref{asy1}, \eqref{asy2} that $w_i$ is non decreasing and satisfies
\begin{align*}
\lim_{s \rightarrow -\infty} w_i(s) = 0, \lim_{s \rightarrow \infty} w_i(s) = \beta_i, 
\lim_{s\rightarrow -\infty} e^{-s}w_i^{\prime}(s) = 0, \int_{-\infty}^{\infty} e^sw_i^{\prime}(s)ds < \infty.
\end{align*}
Therefore using the equation \eqref{ne1} we see that $w_i$ satisfies
\begin{align} \label{ne2}
 w_i^{\prime \prime}(s) = w_i^{\prime}(s) \left[2 - \frac{1}{2\pi}\sum_{j=1}^n a_{ij}w_j(s) - e^s\right].
\end{align}
Summing over all $i$ we can rewrite \eqref{ne2} as
\begin{align} \label{ne3}
 \left(\sum_{i=1}^n w_i^{\prime}(s)\right)^{\prime} = 
 \left[2\sum_{i=1}^nw_i(s) - \frac{1}{4\pi}\sum_{i=1}^n\sum_{j=1}^n a_{ij}w_i(s)w_j(s)\right]^{\prime} -
 \sum_{i=1}^ne^sw_i^{\prime}(s).
\end{align}
Since $\lim_{s \rightarrow \infty}\sum_{i=1}^n w_i(s) = \sum_{i=1}^n \beta_i,w_i$ are non decreasing we can find a
sequence $s_m$ converging to $\infty$ such that $\sum_{i=1}^n w_i^{\prime}(s_m) \rightarrow 0$ as $m \rightarrow \infty.$
 Therefore integrating \eqref{ne3} from $-\infty$ to $s_m$ and letting $m \rightarrow \infty$ we obtain
 \begin{align*}
  2\sum_{i=1}^n \beta_i - \frac{1}{4\pi} \sum_{i=1}^n \sum_{j=1}^na_{ij}\beta_i\beta_j = \sum_{i=1}^n\int_{-\infty}^{\infty}
  e^sw_i^{\prime}(s) ds
\end{align*}
which implies $\Lambda_I(\bm{\beta})>0,$ contradicting our assumption. This completes the proof of the corollary.
 \end{proof}
 \label{Bibliography}

\bibliography{Quadratic_Potentialnew} 

\bibliographystyle{alpha} 

\end{document}